\documentclass{article}
\usepackage[utf8]{inputenc}
\usepackage{amsmath, amsthm, amssymb, amsfonts, wasysym, stmaryrd, mathtools}
\usepackage{tikz, tikz-cd}
\usepackage{enumerate, enumitem}
\usepackage{graphicx}
\usepackage{courier}
\usepackage{mathrsfs}
\usepackage{hyperref}
\usepackage{verbatim}

\usepackage[a4paper, total={6in, 8in}]{geometry}

\newcommand{\N}{\mathbb{N}} 
 
\newcommand{\Z}{\mathbb{Z}} 
\newcommand{\Q}{\mathbb{Q}}

\newcommand{\id}{\operatorname{id}}

\newcommand{\fin}{\operatorname{fin}}

\newcommand{\mc}[1]{\mathcal{#1}}

\newcommand{\edge}[3]{{\displaystyle#1\operatornamewithlimits{---}^{#2}#3}}
\newcommand{\halfedge}[2]{{\displaystyle#1\operatornamewithlimits{---}^{#2}}}

\def\XXint#1#2#3{{\setbox0=\hbox{$#1{#2#3}{\int}$}
\vcenter{\hbox{$#2#3$}}\kern-.5\wd0}}

\makeatletter
\newtheoremstyle{custom-definition}
    {3pt} 
    {3pt} 
    {\addtolength{\@totalleftmargin}{0em}
     \addtolength{\linewidth}{-0em}
     \parshape 1 0em \linewidth} 
    {} 
    {\bfseries} 
    {.} 
    {.5em} 
    {} 
\makeatother

\makeatletter
\newtheoremstyle{custom-remark}
    {3pt} 
    {3pt} 
    {\addtolength{\@totalleftmargin}{0.75em}
     \addtolength{\linewidth}{-0.25em}
     \parshape 1 0.75em \linewidth} 
    {} 
    {\bfseries} 
    {.} 
    {.5em} 
    {} 
\makeatother

\newtheorem{theorem}{Theorem}[section]
\newtheorem{lemma}[theorem]{Lemma}
\newtheorem{corollary}[theorem]{Corollary}
\newtheorem{definition}[theorem]{Definition}
\newtheorem{proposition}[theorem]{Proposition}

\theoremstyle{remark}
\newtheorem{example}[theorem]{Example}
\newtheorem{remark}[theorem]{Remark}

\numberwithin{equation}{section}
\numberwithin{theorem}{section}

\setenumerate[1]{label={\textnormal{(\roman*)}}}

\title{\'Etale equivalence relations with certain prescribed torsion in their homology}
\author{Michael Francesco Ala, Hung-Chang Liao, and Aaron Tikuisis}

\begin{document}

\maketitle

\begin{abstract}
Given a non-cyclic simple dimension group $D$ and a subgroup $E$ of $\Q/\Z$, we produce a minimal \'etale equivalence relation $\mc R$ such that $H_0(\mc R)\cong D \oplus E$, where $H_0(\mc R)$ denotes the zeroth homology group of $\mc R$.
The equivalence relation $\mc R$ arises by combining tail-equivalence on a Bratteli diagram with a partial homeomorphism.
\end{abstract}

This paper contains a construction of minimal \'etale equivalence relations which realize the direct sum of a simple dimension group and a subgroup of $\Q/\Z$, as their zeroth homology groups.
The zeroth homology group of an \'etale equivalence relation is an invariant which has been around for a long time, often by other names (such as the ``$D$-invariant'' or ``coinvariant'').
It is the enveloping group of the ``type semigroup'', which goes back to Tarski.
The realization of simple dimension groups via tail equivalence on the path space of a Bratteli diagram is well-known, going back to the classical work of Bratteli \cite{Bratteli72}, Elliott \cite{Elliott76}, Krieger \cite{Krieger79}, Vershik \cite{Vershik81,Vershik82}, and Herman--Putnam--Skau \cite{HermanPutnamSkau}, and Giordano--Putnam--Skau \cite{GiordanoPutnamSkau95}, for example.

The presence of torsion in the zeroth homology group has been found in the past, and has been particularly surprising in the context of equivalence relations coming from tilings \cite{GahlerHuntonKellendonk,ForrestHuntonKellendonk}.
In \cite{Matui08}, Matui obtained fairly explicit calculations of the zeroth homology group, for particular examples containing torsion arising from $\Z^2$-actions on the Cantor set.
Our result fully calculates the zeroth homology group (and even the type semigroup), for particular examples.
The examples are designed to obtain specific groups which contain torsion.

The construction takes a Bratteli diagram for the dimension group, splits it to make room for an additional summand in the zeroth homology, then augments the tail-equivalence relation by a homeomorphism designed to capture the desired torsion.

We further explore a notion of ``approximately inner flip'' for \'etale equivalence relations, and demonstrate that some of our examples have this property.

The paper is organized as follows.
After Section~\ref{sec:Prelim} on preliminaries, we give two concrete examples of our construction.
The first obtains $\Z[\frac12]\oplus \Z/2\Z$, and demonstrates the key idea used to obtain torsion; the second obtains $(\Z+\Z\frac{1+\sqrt5}2) \oplus \Z/2\Z$, and demonstrates how a simple dimension group can be incorporated into the construction.
We set up for the general construction in Section~\ref{sec:Splitting}, containing a Bratteli diagram construction by splitting a node at each layer, and in Sections~\ref{sec:TEBR}, \ref{sec:TEBRetale} containing a method of producing homeomorphisms of path spaces and an associated \'etale equivalence relation.
Section~\ref{sec:Construction} gives our main construction.
Finally, we explore the notion of approximately inner flip in Section~\ref{sec:AIF}.

\section{Preliminaries}
\label{sec:Prelim}

\subsection{\'Etale groupoids, the type semigroup, and homology}

For a groupoid $\mc G$, we denote the unit space by \(\mc G^{(0)}\), the range and source maps by \(r, s : \mc G\to \mc G^{(0)}\) respectively.
A \emph{bisection} is a subset $\gamma$ of $\mc G$ such that the range and the source maps are each injective when restricted to $\gamma$.
An \textit{\'etale groupoid} is a locally compact Hausdorff groupoid \(\mc G\) such that the range map \(r : \mc G\to \mc G^{(0)}\) is a local homeomorphism; in this case, the source map is also a local homeomorphism.
An \emph{\'etale equivalence relation} is an \'etale groupoid $\mc R$ that is principal (i.e., such that if $r(g)=s(g)$ then $g\in \mc R^{(0)}$). 
As a set (but not topologically), we can view such $\mc R$ as a subset of $\mc R^{(0)} \times \mc R^{(0)}$, and we'll use the notation $(x,y)$ to denote elements of the set $\mc R$.
For $U\subseteq \mc R^{(0)}$, write $\id_U\coloneqq \{(x,x): x \in U\}$, a bisection in $\mc R$.

The \'etale groupoids of interest here will be those with Cantor unit spaces (these are often called \emph{ample groupoids} in the literature).
The topology of such an \'etale groupoid has a basis of compact open bisections (this follows from the proof of \cite[Lemma 8.4.9]{SimsSzaboWilliams}, for example).

The following will help us to define \'etale equivalence relations.

\begin{proposition}[{cf.\ \cite[Theorem~2.1]{Putnam18}}]\label{prop:EtaleCriterion}
Let $\mc R$ be an equivalence relation on a Cantor space $X$.
Let $\Gamma$ be a family of subsets of $\mc R$ satisfying the following.
\begin{enumerate}
\item\label{it:EtaleCriterion1} Each element $\gamma$ of $\Gamma$ is a bisection with $r(\gamma)$ and $s(\gamma)$ clopen, and $\gamma$ is the graph of a continuous function from $r(\gamma)$ to $s(\gamma)$.
\item\label{it:EtaleCriterion2} $\{U\subseteq X: U\text{ clopen, }\mathrm{id}_U \in \Gamma\}$ is a basis for the topology on $X$.
\item\label{it:EtaleCriterion3}  For all $\gamma \in \Gamma$, we have $\gamma^{-1} \in \Gamma$.
\item\label{it:EtaleCriterion4} If $\gamma_1,\gamma_2 \in \Gamma$, then 
\begin{equation} \gamma_1\gamma_2 \coloneqq \{(x,z): (x,y)\in \gamma_1,\text{ and }(y,z)\in\gamma_2\text{ for some }y\in X\} \end{equation}
is also in $\Gamma$.
\item\label{it:EtaleCriterion5} For $\gamma_1,\gamma_2 \in \Gamma$ and $(x,y) \in \gamma_1 \cap \gamma_2$, there exists $\gamma_3 \in \Gamma$ such that $(x,y) \in \gamma_3 \subseteq \gamma_1 \cap \gamma_2$.
\item \label{it:EtaleCriterion6} $\mc R = \bigcup_{\gamma \in \Gamma} \gamma$.
\end{enumerate}
    Then $\Gamma$ is a basis for a topology on $\mc R$, and $\mc R$ is an \'etale equivalence relation under this topology.

The condition \ref{it:EtaleCriterion5} can be replaced by 
\begin{enumerate}
\item[(v')] For $\gamma \in \Gamma$ and $x \in X$ such that $(x,x) \in \gamma$, there exists a neighbourhood $U$ of $x$ in $X$ such that $\mathrm{id}_U \subseteq \gamma$.
\end{enumerate}
\end{proposition}

\begin{proof}
Let $\Gamma'$ be the set of all finite intersections of elements in $\Gamma$.
By \ref{it:EtaleCriterion5}, every element of $\Gamma'$ can be written as a union of elements in $\Gamma$.
We will verify the conditions in \cite[Theorem~2.1]{Putnam18}, for $\Gamma'$, in order to conclude that it (and therefore also $\Gamma$) is the basis of an \'etale equivalence relation topology.

Conditions 1-4 of \cite[Theorem~2.1]{Putnam18} match conditions \ref{it:EtaleCriterion1}-\ref{it:EtaleCriterion4}, and in each case, we can see that since it holds for $\Gamma$, it also holds for $\Gamma'$.
Condition 5 of \cite[Theorem~2.1]{Putnam18} is that $\Gamma'$ is closed under intersections, which holds by construction.
Therefore, $\Gamma'$ is a basis for an \'etale topology on $\bigcup_{\gamma \in \Gamma'} \gamma = \bigcup_{\gamma \in \Gamma} \gamma = \mc R$.

For the final statement, suppose that \ref{it:EtaleCriterion1}-\ref{it:EtaleCriterion4} and (v') holds and let us check that \ref{it:EtaleCriterion5} holds.
Let $\gamma_1,\gamma_2 \in \Gamma$ and let $(x,y) \in \gamma_1 \cap \gamma_2$.
Then $\gamma_1\gamma_2^{-1} \in \Gamma$ and $(x,x) \in \gamma_1\gamma_2^{-1}$, so by (v') there exists a neighbourhood $U$ of $x$ such that $\mathrm{id}_U\subseteq \gamma_1\gamma_2^{-1}$.
By (ii), we may assume that $U$ is clopen and $\mathrm{id}_U \in \Gamma$.
Hence $\gamma_3 \coloneqq \id_U \gamma_2 \in \Gamma$ by \ref{it:EtaleCriterion4}, and we see that it both contains $(x,y)$ and is contained in $\gamma_1 \cap \gamma_2$.
\end{proof}

Given a Cantor space $X$ and a partial homeomorphism $\varphi:A \to B$ (where $A,B\subseteq X$ are clopen subsets), it generates an equivalence relation $\mc R_\varphi$ via $x \sim \varphi(x)$.
For our construction, we will only need the case where $A$ and $B$ are disjoint; this case simplifies the description of the equivalence relation, and ensures that it has an \'etale topology.
If $A$ and $B$ are disjoint, then
\begin{equation}\label{eq:Rphidesc}
 \mc R_\varphi = \{(x,x): x \in X\} \cup \{(x,\varphi(x)): x \in A\} \cup \{(\varphi(x),x): x \in A\}, \end{equation}
and the \'etale topology has a basis of clopen bisections of the form
\begin{equation} \id_U, \{(x,\varphi(x)): x \in U\}, \{(\varphi(x),x): x \in U\}, \end{equation}
where in the first case, $U$ can be any clopen subset of $X$, and in the other cases, $U$ can be any clopen subset of $A$.

Another construction we will use is combining two equivalence relations.
Given two \'etale equivalence relations $\mathcal R_1,\mathcal R_2$ on a common unit space $X$, we can form the smallest equivalence relation $\mathcal R$ generated by $\mathcal R_1$ and $\mathcal R_2$.
One hopes to topologize this by a basis $\Gamma$ of bisections of the form 
\begin{equation} \label{eq:CombinedEquivRelation} \gamma_1\cdots\gamma_n \coloneqq \{(x_0,x_n): (x_0,x_1) \in \gamma_1, \cdots (x_{n-1},x_n) \in \gamma_n\}, \end{equation}
where $n\in \mathbb N$ and each $\gamma_i$ is a compact open bisection in $\mc R_1$ or $\mc R_2$.
The following tells us when this is a basis making $\mc R$ an \'etale equivalence relation.

\begin{proposition}\label{prop:EtaleCombinedRelation}
    Let $\mc R_1,\mc R_2$ be \'etale equivalence relations and define $\Gamma$ as above.
Suppose that for any $\gamma \in \Gamma$ and any $x \in X$ such that $(x,x) \in \gamma$, there exists a neighbourhood $U$ of $x$ in $X$ such that $\id_U \subseteq \gamma$ for all $y\in U$.
  Then $\Gamma$ is a basis for a topology on $\mc R$, and $\mc R$ is an \'etale equivalence relation under this topology.
\end{proposition}

\begin{proof}
The hypothesis in the last line is exactly condition (v') of Proposition~\ref{prop:EtaleCriterion}; we need only check that the other hypotheses hold automatically.
Conditions \ref{it:EtaleCriterion3}, \ref{it:EtaleCriterion4}, and \ref{it:EtaleCriterion6} are clear.
For \ref{it:EtaleCriterion1}, consider a set $\gamma$ of the form \eqref{eq:CombinedEquivRelation}, for some $\gamma_1,\dots,\gamma_n$.
Since each $\gamma_i$ is the graph of a continuous injective function, so is $\gamma$.
Moreover, letting $\gamma' \coloneqq \gamma_2\dots\gamma_n$ and treating $\gamma_1$ as a function $s(\gamma_1) \to r(\gamma_1)\subseteq X$ in the following, we have
\begin{equation} s(\gamma)=\gamma_1^{-1}(s(\gamma')). \end{equation}
Using induction, we may assume that $s(\gamma')$ is clopen, and then using that $\gamma_1$ is continuous with $s(\gamma_1)$ clopen, it follows that $s(\gamma)$ is clopen.
Likewise, $r(\gamma)$ is clopen.
By \ref{it:EtaleCriterion1}, if $\id_U \in \Gamma$ then $U$ is clopen.
From this and Definition~\ref{eq:CombinedEquivRelation}, \ref{it:EtaleCriterion2} is clear.
\end{proof}

Groupoid cohomology for \'etale groupoids over zero-dimensional spaces is defined by Crainic and Moerdijk in \cite{CrainicMoerdijk} (in a more general setting of sheaves over groupoids). We refer to Matui's definition in \cite[Definition~3.1]{Matui12} which is more palatable by avoiding sheaves.

\begin{definition}[\cite{CrainicMoerdijk}]
    Let \(\mc G\) be an \'etale groupoid with Cantor unit space. The \textbf{zeroth homology group} \(H_0(\mc G)\) -- also called the \textit{covariant} -- is the quotient of $C_c(\mc G^{(0)},\mathbb Z)$ by the set of functions of the form
\begin{equation}
\label{eq:MatuiHomologyFunc}
 F(x) := \sum_{s(g)=x} k(g) - \sum_{r(g)=x} k(g), \end{equation}
where $k \in C_c(\mathcal G,\mathbb Z)$.
\end{definition}

A related notion is the type semigroup, which goes back to Tarski.
It is formalized in \cite{BonickeLi,RainoneSims}.
There it is defined as a quotient of $C_0(\mc G^{(0)},\Z_+)$, where $\Z_+ \coloneqq \{0,1,\dots,\}$.
Observe that this semigroup $C_0(\mc G^{(0)},\mathbb Z_+)$ is readily identified with the free abelian group with generators $\langle A\rangle$ where $A$ is a compact open subset of $\mc G^{(0)}$, satisfying the relation $\langle A \rangle + \langle B \rangle = \langle A \cup B\rangle$ provided $A \cap B = \emptyset$.
Using this observation, we can see that the following definition agrees with the ones in \cite{BonickeLi,RainoneSims}.

\begin{definition}\label{def:TypeSemigroup}
    Let \(\mc G\) be an \'etale groupoid with Cantor unit space. The \emph{type semigroup} of $\mc G$, denoted $S(\mc G)$, is the universal abelian semigroup with generators $[A]$ where $A$ is a compact open set in $\mc G^{(0)}$, satisfying the relations
\begin{enumerate}
    \item $[s(U)]=[r(U)]$ for any compact open bisection $U\subseteq \mc G$, and
    \item $[A]+[B]=[A\cup B]$ provided $A\cap B = \emptyset$.
\end{enumerate}
\end{definition}

\begin{proposition}[{\cite[Lemma 1.5]{AraBonickeBosaLi}}]
\label{prop:HomologyEnvelopingTypeSemigroup}
    Let \(\mc G\) be an \'etale groupoid with Cantor unit space.
    Then $H_0(\mc G)$ is the enveloping group of $S(\mc G)$.
\end{proposition}

The image of $S(\mc G)$ in $H_0(\mc G)$ is the \emph{positive cone} of $H_0(\mc G)$, denoted $H_0(\mc G)_+$.

The following provides a natural way of defining homomorphisms between type semigroups.

\begin{proposition}\label{prop:AbstractHomomorphism}
Let $\mathcal G,\mathcal H$ be \'etale groupoids with Cantor unit spaces, and let $\theta:\mathcal G^{(0)} \to \mathcal H^{(0)}$ be a homeomorphism.
Suppose that for any compact open bisection $U\subseteq \mathcal G$, there exists a compact open bisection $V \subseteq \mathcal H$ such that
\begin{equation} \theta(s_{\mc G}(U))=s_{\mc H}(V)\quad\text{and}\quad \theta(r_{\mc G}(U))=r_{\mc H}(V). \end{equation}
Then $\theta$ induces a surjective homomorphism $S(\mc G) \to S(\mc H)$ given by $\theta_*([A]) \mapsto [\theta(A)]$.
\end{proposition}

\subsection{Dimension groups and Bratteli diagrams}

Recall that a \textit{partially ordered abelian group} is an abelian group \(G\) equipped with a \emph{cone of positive elements} $G_+$ (which induces a partial order), and a morphism of partially ordered abelian groups is an order-preserving group homomorphism.
For us a cone includes $0$.

\begin{definition}\label{def:dimensiongroup}
    A \textbf{dimension group} is an inductive limit, in the category of partially ordered abelian groups, of a sequence of groups of the form $(\Z^n,\Z_+^n)$.
\end{definition}

A well-known theorem of Effros, Handelman, and Shen says that dimension groups are \textit{precisely} those partially ordered abelian groups \((G,G^+)\) which are both \textit{unperforated} and possess the \textit{Riesz interpolation property} (\cite{EffrosHandelmanShen}).

Positive homomorphisms from $(\Z^n,\Z_+^n)$ to $(\Z^m,\Z_+^m)$ are described by matrices in $M_{m\times n}(\Z_+)$; we will write the connecting maps in an inductive system for a dimension group as such matrices.
If a dimension group is simple (i.e., has no order ideals; see \cite[Chapter 14]{Goodearl:book}), then it is an inductive limit of a system as in Definition~\ref{def:dimensiongroup} where, in addition, all the matrix entries are nonzero in all the connecting maps (\cite[Theorem~2.14]{Putnam10}).

\begin{definition}
    A \textit{Bratteli diagram} is a directed multigraph \(\mc B = (V, E)\) such that
\begin{enumerate}
    \item the vertices are graded,
\begin{equation} V = V_{\ell_0}\sqcup V_{{\ell_0}+1}\sqcup \cdots, \end{equation}
for some ${\ell_0}\in \mathbb N$ (when not otherwise specified, ${\ell_0}=0$) with each $V_\ell$ finite, and
\item the edges are graded,
\begin{equation} E = E_{\ell_0} \sqcup E_{{\ell_0}+1} \sqcup \cdots \end{equation}
with each $E_\ell$ finite and containing only edges from vertices in $V_\ell$ to vertices in $V_{\ell+1}$.
\end{enumerate}
\end{definition}
To help make things clear, the edges in our Bratteli diagrams will always be of the form $\edge{v}{e}{w}$, where $v$ and $w$ are the vertices that the edge starts and ends at, respectively, and $e$ is a label.
We write $E(v,w)$ for the set of labels on edges from $v$ to $w$.
A path containing edges $\edge{v_\ell}{e_\ell}{v_{\ell+1}}, \edge{v_{\ell+1}}{e_{\ell+1}}{v_{\ell+2}}, \dots$ is written 
\begin{equation} \halfedge{v_\ell}{e_\ell}\halfedge{v_{\ell+1}}{e_{\ell+1}}\halfedge{v_{\ell+2}}{}\cdots.\end{equation}

Bratteli diagrams are typically used to represent an inductive system in the definition of a dimension group.
More precisely, if $\mathcal B=(V,E)$ is a Bratteli diagram, then one forms an inductive system
\begin{equation} \label{eq:BrattDiagSystem} \begin{tikzcd}
    G_1 \ar[r,"A_1^2"] & G_2 \ar[r,"A_2^3"] & \cdots
\end{tikzcd}\end{equation}
by setting $G_\ell\coloneqq \mathbb Z^{V_\ell}$, and the morphism $A_{\ell}^{\ell+1}$ is the matrix $[a^{\ell}_{w,v}] \in M_{V_{\ell+1}\times V_{\ell}}(\mathbb Z_+)$ with $a_{v,w}\coloneqq |E(v,w)|$ (that is, the number of edges from $v$ to $w$).
Naturally, every inductive system of ordered groups of the form $(\mathbb Z^n,\mathbb Z_+^n)$ comes from a Bratteli diagram, by reversing this construction.

Given a Bratteli diagram $\mathcal B=(V,E)$ and given $\ell \geq k \geq \ell_0$, we write $E_{k\dots\ell}$ (or $E_{k\dots\ell}(\mc B)$) for the set of paths from a vertex in $V_k$ to a vertex in $V_\ell$ (if $k=\ell$ then $E_{k\dots\ell}$ is just $V_\ell$).
For $m\geq \ell \geq k \geq \ell_0$ and $A \subseteq E_{k\dots\ell}$, we write $AE_{\dots m}$ (or $AE_{\dots m}(\mc B)$) for the subset of $E_{k\dots m}$ consisting of paths that begin with a path in $A$; in the case of a singleton, we write $wE_{\dots m}$ to mean $\{w\}E_{\dots m}$.

The \textit{infinite path space} asociated to a Bratteli diagram \(\mc B\), denoted \(X_{\mc B}\), consists of all infinite paths in \(\mc B\) starting at a vertex in \(V_{\ell_0}\).
For a set of finite paths $A \subseteq E_{\ell_0 \dots\ell}$, we write $AX_{\mc B}$ for the subset of $X_{\mc B}$ containing infinite paths that start with a path in $A$; in the case of a singleton, we write $wX_{\mc B}$ to mean $\{w\}X_{\mc B}$, and this is called a \emph{cylinder set}.

$X_{\mc B}$ is endowed with the topology given by using the cylinder sets $wX_{\mc B}$ (over all finite paths $w$ starting in $V_0$) as a basis.
Cylinder sets are compact and open in this topology, and $X_{\mc B}$ is totally disconnected, compact, and metrizable.

Given a Bratteli diagram $\mc B=(V=V_{\ell_0}\sqcup V_{\ell_0+1}\sqcup\cdots ,E=E_{\ell_0}\sqcup E_{\ell_0+1}\sqcup\cdots)$ and a natural number $k_0\geq \ell_0$, define the \emph{truncated Bratteli diagram} 
\begin{equation} \mathrm{Trunc}_{k_0}(\mathcal B) \coloneqq (V_{k_0}\sqcup V_{k_0+1}\sqcup\cdots, E_{k_0}\sqcup E_{k_0+1}\sqcup\cdots). \end{equation}
For $k\geq k_0\geq \ell \geq \ell_0$, with a slight abuse of notation, define $\mathrm{Trunc}_{k_0}:E_{\ell\dots m} \to E_{k_0\dots m}$ by 
\begin{equation} \mathrm{Trunc}_{k_0}(\edge{v_{\ell}}{e_{\ell}}{v_{\ell+1}}\cdots\edge{v_{m-1}}{e_{m-1}}{v_m} \in E_{\ell\dots m})\coloneqq \edge{v_{k_0}}{e_{k_0}}{v_{k_0+1}}\cdots\edge{v_{m-1}}{e_{m-1}}{v_m} \in E_{k_0\dots m}. \end{equation}
(The image can be viewed as a path in $\mathrm{Trunc}_{k_0}(\mc B)$.)
Similarly, for an infinite path $w=\halfedge{v_{\ell_0}}{e_{\ell_0}}\halfedge{v_{\ell_0+1}}{e_{\ell_0+1}}\cdots \in X_{\mc B}$, define
\begin{equation} \mathrm{Trunc}_{k_0}(w)\coloneqq \halfedge{v_{k_0}}{e_{k_0}}\halfedge{v_{k_0+1}}{e_{k_0+1}}\cdots \in X_{\mathrm{Trunc}_{k_0}(\mc B)}. \end{equation}

We let \(\mc R_{\mc B,\mathrm{tail}}\) be the tail-equivalence relation on \(X_{\mc B}\), and for any $\ell\geq \ell_0$, the subequivalence ${\mc R_{\mc B,\ell}}$ consists of all pairs $(w,z)$ such that $\mathrm{Trunc}_\ell(w)=\mathrm{Trunc}_\ell(z)$.
Thus, \(\mc R_{\mc B,\mathrm{tail}} = \bigcup_\ell \mc R_{\mc B,\ell}\).
We endow the tail-equivalence relation with the inductive limit topology, which has a basis of compact open bisections of the form
\begin{equation}
\label{eq:TailEqBasis}
U_{w,z}\coloneqq \{(wx,zx) \in X_{\mc B} \times X_{\mc B}: x\text{ is an infinite path starting at $v$}\}, \end{equation}
where $w,z$ are two finite paths starting in $V_{\ell_0}$ and ending at the same vertex $v$ (see \cite[Section~2.3]{Putnam10} for example).

Homology provides a formal link between tail equivalence of a Bratteli diagram and the corresponding dimension group.
The argument given in \cite[Theorem~2.9]{Putnam10}, which is stated in terms of the zeroth homology group, in fact shows the following.

\begin{proposition}[{cf.\ \cite[Theorem~2.9]{Putnam10}}]
Let $\mathcal B=(V,E)$ be a Bratteli diagram and let $(G,G_+)$ be the dimension group given as the inductive limit of the associated inductive sequence \eqref{eq:BrattDiagSystem}.
Then 
\begin{equation} S(\mc R_{\mc B,\mathrm{tail}}) \cong G_+, \end{equation}
by an explicit isomorphism $\alpha$ described as follows.
Let $w$ be a path from a vertex in $V_0$ to a vertex $v_\ell \in V_\ell$.
Let $e_{v_\ell} \in \Z_+^{V_\ell}$ be the canonical basis element that is $0$ in every entry except for $v_\ell$, where it is $1$.
Then $\alpha([wX_{\mc B}])$ is equal to the image of $e_{v_\ell}$ in $G_+$.
\end{proposition}

In our construction, we will make use of a particular type of graph homomorphism between Bratteli diagrams.
We informally call $T:\mc B \to \mc B'$ satisfying the hypotheses of the following lemma a \emph{shadow} homomorphism of Bratteli diagrams, and say that $\mc B'$ is a shadow of $\mc B$.

\begin{lemma}[Bratteli diagram shadows] \label{lem:Shadow}
    Let $\mc B=(V,E),\mc B'=(V',E')$ be Bratteli diagrams.
    Let $T:\mc B \to \mc B'$ be a graph homomorphism, such that:
    \begin{enumerate}
        \item $T$ maps the first level $V_{\ell_0}$ in $\mc B$ bijectively to the first level $V'_{\ell_0}$ in $\mc B'$;
        \item for $\ell \geq 1$, $T$ maps a level $V_\ell$ in $\mc B$ surjectively onto the corresponding level $V'_\ell$ in $\mc B'$; and
        \item for every vertex $v \in V$, $T$ maps the set of edges starting at $v$ bijectively to the set of edges starting at $T(v)$.
    \end{enumerate}
    Then $T$ induces a continuous injective groupoid morphism $s_{\mathrm{tail}}:\mc R_{\mc B,\mathrm{tail}} \to \mc R_{\mc B',\mathrm{tail}}$, which restricts to a homeomorphism $s_X:X_{\mc B} \to X_{\mc B'}$ of unit spaces.
\end{lemma}

\begin{proof}
    From the hypothesis, we can, by induction on the path length, see that $T$ is bijective on finite paths, and therefore also on infinite paths.
    This proves that $s_X$ is bijective, and continuity of $s_X$ and its inverse are clear.

    We can see that $s_X \times s_X$ maps $\mc R_{\mc B,\ell}$ into $\mc R_{\mc B',\ell}$ for all $\ell$, and therefore it induces an injective groupoid morphism $s_{\mathrm{tail}}$.
    For continuity, consider an element $(a,b) \in \mc R_{\mc B,\mathrm{tail}}$ and a basis set $U_{w',z'}$ as in \eqref{eq:TailEqBasis} for $\mc R_{\mc B',\mathrm{tail}}$ containing $(s_X\times s_X)(a,b)$, where $w',z'$ are two finite paths starting in $V_0'$ and ending at the same vertex.
    Then there are unique paths $w$, $z$ in $\mc B$ such that $w'=T(w)$ and $z'=T(z)$, and consequently $w$ and $z$ are the prefixes of $a$ and $b$ respectively.
    Since $a,b$ are tail-equivalent, and because the basis set $U_{w',z'}$ will only shrink if we make $w',z'$ longer, we may possibly lengthen $w$ and $z$ so that $a=wx$ and $b=zx$ for some infinite path $x$ in $\mc B$.
    Then we see that $(a,b) \in U_{w,z}$ and $(s_X \times s_X)(U_{w,z}) \subseteq U_{w',z'}$, as required.
\end{proof}

\section{Illustrative examples}
\label{sec:Examples}
To illustrate our construction, we will demonstrate it with two examples.
We will not prove all claims in these examples, as we do this carefully in the general case.

\subsection{Obtaining the homology group \(\Z[\frac12] \oplus \Z/ 2\Z\)}
First, we begin with a construction to realize the group $\Z[\frac12] \oplus \Z/ 2\Z$ as its zeroth homology group.
(Note that the coincidence of $2$-torsion and the $2$-divisibility neither enables nor even simplifies the construction here.)

The Bratteli diagram we will use is $\mathcal B$ on the left below; $\mathcal B_{2^\infty}$ is on the right, which is a shadow of $\mathcal B$ in the sense of Lemma~\ref{lem:Shadow}
\begin{equation} \begin{tikzcd}[column sep=2mm,row sep=0mm]
    & \mc B & && \mc B_{2^\infty} \\
    & \bullet \ar[ddl,no head] \ar[ddl,shift right, no head, swap] \ar[ddr,no head,swap] \ar[ddr,shift left,no head] &
    &\qquad& \bullet \ar[dd,no head,shift right=0.5] \ar[dd,shift right=1.5, no head] \ar[dd,no head,shift left=0.5] \ar[dd,no head,shift left=1.5] \\ \ \\
    a \ar[ddrr,no head,pos=0.2] \ar[ddrr,shift right,no head,swap,pos=0.2] && b \ar[ddll,no head] \ar[dd,no head]
    &&\bullet \ar[dd,no head,shift right=0.5]\ar[dd,no head,shift left=0.5] \\ \  \\
    a \ar[ddrr,no head,pos=0.2] \ar[ddrr,shift right,no head,swap,pos=0.2] && b \ar[ddll,no head] \ar[dd,no head]&
    &\bullet \ar[dd,no head,shift right=0.5]\ar[dd,no head,shift left=0.5] \\ \ \\
    \phantom{a} && \phantom{b} & & \phantom{\bullet} \\
    & \vdots & &
    &\vdots
\end{tikzcd} \end{equation}
We reuse vertex labels to avoid notational clutter.
Likewise we have not included edge labels, but when describing paths we use the labels $1$ and $2$ whenever there are two edges between vertices.

Tail-equivalence on $\mc B_{2^\infty}$ gives $\Z[\frac12]_+$ as the type semigroup.
When we project a path from $\mathcal B$ to $\mc B_{2^\infty}$, what is forgetten is, essentially, the labels on the vertices ($a$ and $b$).
We would like $a$-cylinders (that is, ones coming from a path that ends at an $a$ vertex) to have class $0$ is the torsion component ($\Z/2\Z$), and $b$-cylinders to have class $1$.
Note that an $a$-cylinder decomposes as two $b$-cylinders, corresponding to the relation $0=1+1$ in $\Z/2\Z$, while a $b$-cylinder decomposes as the union of an $a$- and a $b$-cylinder, corresponding to $1=0+1$.
Tail-equivalence alone will not impose all the relations required to get $\Z[\frac12]\oplus \Z/2\Z$ as the zeroth homology group (indeed, dimension groups are always torsion-free).
We shall augment the tail-equivalence relation, adding a relation that makes the union of two $a$-cylinders equivalent to the union of two $b$-cylinders \emph{at the same level}.

Set
\begin{equation} \begin{split} A&\coloneqq (\edge{\bullet}1{a}X_{\mc B}) \sqcup (\edge{\bullet}2{a}X_{\mc B}), \\
B&\coloneqq (\edge{\bullet}1{b}X_{\mc B}) \sqcup (\edge{\bullet}2{b}X_{\mc B}). \end{split} \end{equation}
We shall define a bijective map $\varphi:A \to B$ by the following diagram, where each box denotes a cylinder defined by the labelled path:
\begin{equation}\begin{split} \includegraphics{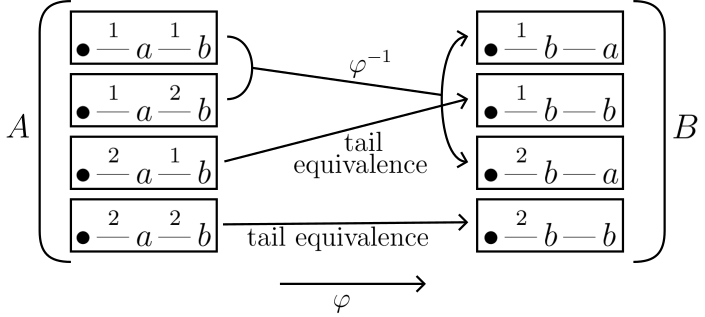}\end{split} \end{equation}
This is a recursive definition; as we now explain.
The subset $(\halfedge{\bullet}{1}\edge{a}{1}{b}X_{\mc B}) \sqcup (\halfedge{\bullet}{1}\edge{a}{2}{b}X_{\mc B})$ identifies canonically with $B$ -- both are the union of two $b$-cylinders.
Likewise, the subset $\halfedge{\bullet}{1}\edge{b}{}{a}X_{\mc B} \sqcup \halfedge{\bullet}{2}\edge{b}{}{b}X_{\mc B}$ identifies with $A$.
We ask that the restriction of $\varphi$ to these pieces, after making these identifications, is equal to $\varphi^{-1}$.
This rule determines $\varphi$ at all points except $\halfedge{\bullet}{1}\halfedge{a}{1}\halfedge{b}{}\halfedge{a}{1}\halfedge{b}{}\cdots$, and determines $\varphi^{-1}$ at all points except $\halfedge{\bullet}{1}\halfedge{b}{}\halfedge{a}{1}\halfedge{b}{}\halfedge{a}{1}\cdots$; we finally define
\begin{equation} \varphi(\halfedge{\bullet}{1}\halfedge{a}{1}\halfedge{b}{}\halfedge{a}{1}\cdots) \coloneqq \halfedge{\bullet}{1}\halfedge{b}{}\halfedge{a}{1}\halfedge{b}{}\halfedge{a}{1}\cdots.\end{equation}

The (topological) equivalence relation $\mc R$ is then generated by tail-equivalence together with $\varphi$.

In order to compute $S(\mc R)$ (and thereby $H_0(\mc R)$), first note that there is a homomorphism $\beta:S(\mc R) \to \Z/2\Z$ which sends classes of $a$-cylinders to $0$ and classes of $b$-cylinders to $1$; the existence of this homomorphism amounts to checking that tail-equivalence, $\varphi$, and decomposition of cylinders respect the relations in $\Z/2\Z$.
We also note that the shadow homeomorphism $s_X:X_{\mc B} \to X_{\mc B_{2^\infty}}$ coming from Lemma~\ref{lem:Shadow} induces a map $S(\mc R) \to S(\mc R_{\mc B_{2^\infty},\mathrm{tail}}) \cong \Z[\frac12]_+$, using Proposition~\ref{prop:AbstractHomomorphism}.
To verify the hypothesis of this proposition, for any cylinder $A \subseteq X_{\mc B}$, the images of $A$ and $\varphi(A)$ in $X_{\mc B_{2^\infty}}$ can be decomposed into sub-cylinders and matched up by tail-equivalence.
This is a bit subtle, however, as the shadow of $\varphi$ is not captured by tail-equivalence at individual points.
Indeed, $s_X(\halfedge{\bullet}{1}\halfedge{a}{1}\halfedge{b}{}\halfedge{a}{1}\cdots)$ is not tail-equivalent to $s_X(\halfedge{\bullet}{1}\halfedge{b}{}\halfedge{a}{1}\halfedge{b}{}\halfedge{a}{1}\cdots)$, though they are equivalent via $\varphi$.

The final part of the computation (whose details we will fully omit in this section) is to check that if $A,B$ are clopen sets in $X_{\mc B}$ such that $[s(A)]=[s(B)]$ in $\Z[\frac12]_+$ and $\beta([A])=\beta([B])$ in $\Z/2\Z$ then $[A]=[B]$ in $S(\mc R)$.

\subsection{Obtaining the homology group $(\Z+\Z\frac{1+\sqrt5}2) \oplus \Z/2\Z$ via the Fibonacci dimension group}

The ``Fibonacci'' dimension group is $D\coloneqq \Z+\Z\frac{1+\sqrt5}2$, with $D_+\coloneqq D \cap [0,\infty)$.
It can be realized as a stationary inductive limit of $\Z^2$ with the connecting map
\begin{equation} \begin{bmatrix}
    1 & 1 \\ 1 & 0
\end{bmatrix}.
\end{equation}
In order to obtain enough edges for our construction, we use a telescoped system, with connecting maps
\begin{equation}
\begin{bmatrix}
    1 & 1 \\ 1 & 0
\end{bmatrix}^3 = \begin{bmatrix}
    3 & 2 \\ 2 & 1
\end{bmatrix}.
\end{equation}
This system gives us the shadow Bratteli diagram on the right below.
\begin{equation} \begin{tikzcd}[column sep=9mm,row sep=0mm]
    &\mc B &&&& \mc B_D & \\
    &\bullet \ar[dddl,no head] \ar[ddd,shift left=0.5,no head]\ar[ddd,shift right=0.5,no head] \ar[dddr,shift left=0.5,no head]\ar[dddr,shift right=0.5,no head] && && \bullet \ar[dddl,no head] \ar[dddr,shift left=0.5,no head]\ar[dddr,shift right=0.5,no head]\ar[dddr,shift left=1.5,no head]\ar[dddr,shift right=1.5,no head]
    \\ \ \\ \ \\
    v_a \ar[ddd,no head] \ar[dddr,shift right=0.5,no head]\ar[dddr,shift left=0.5,no head] & 
    w_a \ar[dddl,shift right=0.5,no head]\ar[dddl,shift left=0.5,no head] \ar[ddd,no head] \ar[dddr,shift right=0.5,no head]\ar[dddr,shift left=0.5,no head] &
    w_b \ar[dddll,shift right=0.5,no head]\ar[dddll,shift left=0.5,no head] \ar[dddl,shift right=0.5,no head]\ar[dddl,shift left=0.5,no head] \ar[ddd,no head] &\quad &
    v \ar[ddd,no head] \ar[dddrr,shift right=0.5,no head]\ar[dddrr,shift left=0.5,no head] && w \ar[dddll,shift right=0.5,no head]\ar[dddll,shift left=0.5,no head] \ar[ddd,shift right,no head]\ar[ddd,shift left,no head]\ar[ddd,no head]\\ \ \\ \ \\
    v_a \ar[ddd,no head] \ar[dddr,shift right=0.5,no head]\ar[dddr,shift left=0.5,no head] & 
    w_a \ar[dddl,shift right=0.5,no head]\ar[dddl,shift left=0.5,no head] \ar[ddd,no head] \ar[dddr,shift right=0.5,no head]\ar[dddr,shift left=0.5,no head] &
    w_b \ar[dddll,shift right=0.5,no head]\ar[dddll,shift left=0.5,no head] \ar[dddl,shift right=0.5,no head]\ar[dddl,shift left=0.5,no head] \ar[ddd,no head] &\quad &
    v \ar[ddd,no head] \ar[dddrr,shift right=0.5,no head]\ar[dddrr,shift left=0.5,no head] && w \ar[dddll,shift right=0.5,no head]\ar[dddll,shift left=0.5,no head] \ar[ddd,shift right,no head]\ar[ddd,shift left,no head]\ar[ddd,no head]\\ \ \\ \ \\
    v_a & w_a & w_b && v && w \\
    & \vdots &&&& \vdots
\end{tikzcd} \end{equation}
Set
\begin{equation} \begin{split} A&\coloneqq (\edge{\bullet}1{w_a}X_{\mc B}) \sqcup (\edge{\bullet}2{w_a}X_{\mc B}), \\
B&\coloneqq (\edge{\bullet}1{w_b}X_{\mc B}) \sqcup (\edge{\bullet}2{w_b}X_{\mc B}). \end{split} \end{equation}
Similarly to the previous example, we define a bijective map $\varphi:A \to B$ recursively, and mostly using tail-equivalence.
In terms of cylinders at the second level (i.e., labelled by paths of length $2$), we have
\begin{itemize}
    \item $A$ is made up of 4 $v_a$-cylinders, 2 $w_a$-cylinders, and 4 $w_b$-cylinders, and
    \item $B$ is made up of 4 $v_a$-cylinders, 4 $w_a$-cylinders, and 2 $w_b$-cylinders.
\end{itemize}
By matching up cylinders in $A$ with like cylinders in $B$ using tail-equivalence, we can define $\varphi$ on all of $A$ except
\begin{equation} \halfedge{\bullet}1\edge{w_a}1{w_b}X_{\mc B} \sqcup \halfedge{\bullet}1\edge{w_a}2{w_b}X_{\mc B}, \end{equation}
and $\varphi^{-1}$ is defined on all of $B$ except
\begin{equation} \halfedge{\bullet}1\edge{w_b}1{w_a}X_{\mc B} \sqcup \halfedge{\bullet}2\edge{w_b}1{w_a}X_{\mc B}. \end{equation}
These remaining parts identify with copies of $B$ and $A$ respectively, and we use recursion to continue to define $\varphi$.
This rule determines $\varphi$ at all points except $\halfedge{\bullet}1\halfedge{w_a}1\halfedge{w_b}1\halfedge{w_a}1\halfedge{w_b}1\cdots$ and determines $\varphi^{-1}$ at all points except $\halfedge{\bullet}1\halfedge{w_b}1\halfedge{w_a}1\halfedge{w_b}1\halfedge{w_a}1\cdots$; we finally define
\begin{equation}  \varphi(\halfedge{\bullet}1\halfedge{w_a}1\halfedge{w_b}1\halfedge{w_a}1\halfedge{w_b}1\cdots)\coloneqq \halfedge{\bullet}1\halfedge{w_b}1\halfedge{w_a}1\halfedge{w_b}1\halfedge{w_a}1\cdots. \end{equation}
The (topological) equivalence relation $\mc R$ is then generated by tail-equivalence together with $\varphi$.

Similarly to the previous example, we have a homomorphism $\beta:S(\mc R) \to \Z/2\Z$ which sends classes of $v_a$- and $w_a$-cylinders to $0$ and classes of $w_b$-cylinders to $1$.
Also, we again have a shadow homomorphism $s_X:X_{\mc B} \to X_{\mc B_D}$, and for any cylinder $A$, $s_X(A)$ is equivalent to $s_X(\varphi(A))$ via tail-equivalence.
So, by Proposition~\ref{prop:AbstractHomomorphism},
$s_X$ induces a map $S(\mc R) \to S(\mc R_{\mc B_D,\mathrm{tail}}) \cong (\Z+\Z\frac{1+\sqrt5}2)\cap [0,\infty)$.
With the final details omitted, these two maps combine to give an isomorphism 
\begin{equation} S(\mc R) \cong \{0\} \cup \Big[(\Z+\Z\tfrac{1+\sqrt5}2) \cap (0,\infty) \oplus \Z/2\Z\Big]. \end{equation}

In the general case, beginning with a Bratteli diagram $\mc B_D$, we will build a new Bratteli diagram $\mc B$ which has $\mc B_D$ as a shadow, by picking a vertex at each level that is appropriately split, creating an ``a'' and ``b'' copy of it (in the above example, the vertex $w$ was duplicated).
A (partial) homeomorphism $\varphi$ will enforce a relation that a suitable number of $a$-cylinders is equivalent to the same number of $b$-cylinders.
A general procedure for obtaining $\mc B$ from $\mc B_D$ is described in Section~\ref{sec:Splitting}, and one for obtaining the homeomorphism $\varphi$ is in Section~\ref{sec:TEBR}.

\section{A Bratteli diagram construction}
\label{sec:Splitting}

In this section we give a procedure for producing a new Bratteli diagram that has a given input Bratteli diagram as its shadow, obtained by ``splitting'' a vertex at each level.

This diagram demonstrates the splitting of a vertex $w_a^\ell$ at the $\ell^\text{th}$ level.

\begin{equation}        \begin{tikzcd}
        \bullet \arrow[rd, shift right] \arrow[rd, red] & \bullet \arrow[d, red] \arrow[d, shift right] \arrow[d, shift left, red] & \bullet \arrow[ld, red] \arrow[ld, shift right] \\
        & w_a^\ell \arrow[ld] \arrow[d] \arrow[rd] \arrow[d, shift left] \arrow[rd, shift right] & \\
        \bullet & \bullet & \bullet
        \end{tikzcd}
        \quad\to\quad
        \begin{tikzcd}
        \bullet \arrow[rd] \arrow[rrrd, red] & \bullet \arrow[d] \arrow[rrd, red] \arrow[rrd, shift right, red] & \bullet \arrow[ld] \arrow[rd, red] & \\
        & w_a^\ell \arrow[ld] \arrow[d] \arrow[d, shift right] \arrow[rd] \arrow[rd, shift right] &  & w_b^\ell \arrow [llld, blue] \arrow[lld, blue] \arrow[lld, blue, shift right] \arrow[ld, blue] \arrow[ld, blue, shift right] \\
        \bullet & \bullet & \bullet &
        \end{tikzcd} \end{equation}

Notice that the \textit{in-edges} are split up (indicated in red), whereas the \textit{out-edges} are duplicated (indicated in blue). There is a choice of which in-edges to split; in the above diagram, this choice is indicated by colouring these edges red in the initial (left) diagram.
(Essentially, what matters is how many edges to split.)
We start from the top of the Bratteli diagram and proceed downwards, splitting nodes at levels $1$ then $2$ and so on.

Before formally defining the splitting, we show an example of it in action.
Consider the stationary Bratteli diagram corresponding to the connecting map $\begin{bmatrix}3 & 2 \\ 2 & 1\end{bmatrix}$. We split the vertices $w_a^1,w_a^2,$ and $w_a^3$, and use boldface to emphasize the node being split at each step.
Unlike the above diagram, we do not colour edges red before they have been split (to avoid ambiguity in the middle steps).

\begin{equation}\begin{split}
    \begin{tikzcd}[ampersand replacement=\&]
    \& \& \bullet \arrow[ddll, shift right] \arrow[ddll] \arrow[dd] \\ \& \& \\
    w_a^1 \arrow[dd, shift right] \arrow[dd] \arrow[dd, shift left] \arrow[rrdd, shift right] \arrow[rrdd] \&  \& v^1 \arrow[lldd, shift right] \arrow[lldd] \arrow[dd] \\ \&  \& \\
    v^2 \arrow[dd, shift right] \arrow[dd, shift left] \arrow[dd] \arrow[rrdd, shift left] \arrow[rrdd]  \&  \& w_a^2 \arrow[lldd, shift left] \arrow[lldd] \arrow[dd]  \\ \&  \& \\ w_a^3 \&  \& v^3
    \end{tikzcd}
    \quad&\to\quad
    \begin{tikzcd}[ampersand replacement=\&]
    \& \& \bullet \arrow[ddll] \arrow[ddrr,red] \arrow[dd] \\ \& \& \\
    \bf{w_a^1} \arrow[dd, shift right] \arrow[dd] \arrow[dd, shift left] \arrow[rrdd, shift right] \arrow[rrdd] \&  \& v^1 \arrow[lldd, shift right] \arrow[lldd] \arrow[dd] \& \& w_b^1 \arrow[lllldd, shift left,blue] \arrow[lllldd,blue] \arrow[lllldd, shift right,blue] \ar[lldd,shift right, blue] \ar[lldd,blue] \\ \& \& \\
    v^2 \arrow[dd, shift right] \arrow[dd, shift left] \arrow[dd] \arrow[rrdd, shift left] \arrow[rrdd]  \&  \& w_a^2\arrow[lldd, shift left] \arrow[lldd] \arrow[dd]  \\ 
    \&  \& \\ 
    w_a^3 \&  \& v^3
    \end{tikzcd}
    \\
    \to
    \begin{tikzcd}[ampersand replacement=\&]
    \& \& \bullet \arrow[ddrr] \arrow[ddll] \arrow[dd] \\ \& \& \\
    w_a^1 \arrow[dd, shift right] \arrow[dd] \arrow[dd, shift left] \arrow[rrrrdd, shift right,red] \arrow[rrrrdd,red] \&  \& v^1 \arrow[lldd, shift right] \arrow[lldd] \arrow[dd] \& \& w_b^1 \arrow[lllldd, shift left] \arrow[lllldd] \arrow[lllldd, shift right] \ar[lldd,shift right] \ar[dd,red] \\ \& \& \\
    v^2 \arrow[dd, shift right] \arrow[dd, shift left] \arrow[dd] \arrow[rrdd, shift left] \arrow[rrdd]  \&  \& \bf{w_a^2} \arrow[lldd, shift left] \arrow[lldd] \arrow[dd] \&\& w_b^2 \ar[lllldd,shift right,blue] \ar[lllldd,blue] \ar[lldd,blue] \\ 
    \&  \& \\ 
    w_a^3 \&  \& v^3
    \end{tikzcd}
    \quad&\to\quad
    \begin{tikzcd}[ampersand replacement=\&]
    \& \& \bullet \arrow[ddrr] \arrow[ddll] \arrow[dd] \\ \& \& \\
    w_a^1 \arrow[dd, shift right] \arrow[dd] \arrow[dd, shift left] \arrow[rrrrdd,  shift left] \arrow[rrrrdd] \& \& v^1 \arrow[lldd, shift left] \arrow[lldd] \arrow[dd] \& \& w_b^1 \arrow[lllldd, shift left] \arrow[lllldd] \arrow[lllldd, shift right] \arrow[lldd, shift right] \\ 
    \& \& \& \& \\
    v^2 \arrow[rrdd, shift right] \arrow[rrdd] \arrow[dd] \arrow[rrrrdd, red] \arrow[rrrrdd, red, shift left] \& \& w_a^2 \arrow[dd] \arrow[rrdd, red] \arrow[rrdd, red, shift left] \& \& w_b^2  \arrow[lldd] \arrow[lllldd] \arrow[dd,red] \\
    \& \& \& \& \\ \bf{w_a^3} \& \& v^3 \& \& w_b^3  
    \end{tikzcd}
\end{split}\end{equation}

For the general construction, the input consists of:
\begin{itemize}
    \item a Bratteli diagram $\mathcal B=(V,E)$ (starting at level $\ell_0=0$);
    \item a choice of vertex $w_a^\ell \in V_\ell$ to split at each level $\ell\geq 1$ --- the duplicated version of $w_a^\ell$ will be denoted $w_b^\ell$; and
    \item a choice of how to split the edges: for each $\ell\geq 0$, sets $F(v^{\ell})\subseteq E(v^{\ell},w_a^{\ell+1})$ for each $v^{\ell} \in V_{\ell}$ and $F(w_b^{\ell}) \subseteq E(w_a^{\ell},w_a^{\ell+1})$. These sets indicate the edge labels to move to the split vertex $w_b^{\ell+1}$, while the other edges stay with $w_a^{\ell+1}$.
    (Note that the edges from $w_b^{\ell}$ to $w_a^{\ell+1}$ and $w_b^{\ell+1}$ are made up of copies of the edges from $w_a^{\ell}$ to $w_a^{\ell+1}$; this is why $F(w_b^{\ell})$ is a subset of $E(w_a^{\ell},w_a^{\ell+1})$.)
\end{itemize}

Given inputs as above, we define the Bratteli diagram $\mathcal C$ 
to have vertex sets $V_{\mc C}^0\coloneqq V_0$ and $V_{\mc C}^\ell \coloneqq V_\ell \sqcup \{w_b^\ell\}$ for $\ell\geq 1$, and edges as follows:
\begin{itemize}
    \item for $v^{\ell} \in V_{\ell}$, $v^{\ell+1} \in V_{\ell+1}\setminus \{w_a^{\ell+1}\}$, and $\edge{v^{\ell}}e{v^{\ell+1}} \in E$, we have the edge $\edge{v^{\ell}}e{v^{\ell+1}}$ in $E_{\mc C}$;
    \item for $v^{\ell} \in V_{\ell}$ and $e \in F(v^{\ell})$, we have the edge $\edge{v^{\ell}}e{w_b^{\ell+1}}$ in $E_{\mc C}$;
    \item for $v^{\ell} \in V_{\ell}$ and $e \in E(v^{\ell},w_a^{\ell+1})\setminus F(v^{\ell})$, we have the edge $\edge{v^{\ell}}e{w_a^{\ell+1}}$ in $E_{\mc C}$;
    \item for $v^{\ell+1} \in V_{\ell+1}\setminus \{w_a^{\ell+1}\}$ and $\edge{w_a^{\ell}}e{v^{\ell+1}}\in E$, we have the edge $\edge{w_b^{\ell}}e{v^{\ell+1}}$ in $E_{\mc C}$;
    \item for $e \in F(w_b^{\ell})$, we have the edge $\edge{w_b^{\ell}}e{w_b^{\ell+1}}$ in $E_{\mc C}$; and
    \item for $e \in E(w_a^{\ell},w_a^{\ell+1})\setminus F(w_b^{\ell})$, we have the edge $\edge{w_b^{\ell}}e{w_a^{\ell+1}}$ in $E_{\mc C}$.
\end{itemize}
Diagramatically we have:
\begin{equation}
    \begin{tikzcd}[ampersand replacement=\&]
    \&\&\&\& v^{\ell} \arrow[ddllll, "E(v^{\ell}{,}v^{\ell+1})" description] \arrow[dd, "F(v^{\ell})" description] \arrow[ddrrrr, "E(v^{\ell}{,}w_a^{\ell+1})\setminus F(v^{\ell})" description] \\ \\
    v^{\ell+1} \&\&\&\& w_b^{\ell+1} \&\&\&\& w_a^{\ell+1}
    \end{tikzcd}
\end{equation}
for $v^{\ell} \in V_{\ell}$ (which includes the case $v^{\ell}=w_a^{\ell}$) and $v^{\ell+1} \in V_{\ell+1}\setminus \{w_a^{\ell+1}\}$; and 
\begin{equation}
    \begin{tikzcd}[ampersand replacement=\&]
    \&\&\&\& w_b^{\ell} \arrow[ddllll, "E(w_a^{\ell}{,}v^{\ell+1})" description] \arrow[dd, "F(w_b^{\ell})" description] \arrow[ddrrrr, "E(w_a^{\ell}{,}w_a^{\ell+1})\setminus F(w_b^{\ell})" description] \\ \\
    v^{\ell+1} \&\&\&\& w_b^{\ell+1} \&\&\&\& w_a^{\ell+1}
    \end{tikzcd}
\end{equation}
for $v^{\ell+1} \in V_{\ell+1}\setminus \{w_a^{\ell+1}\}$.

\begin{proposition}
\label{prop:SplittingShadow}
    Let $\mc B, (w_a^\ell)_\ell, F,$ and $\mc C$ 
be as above.
    Then there is a shadow graph homomorphism $T:\mc C \to \mc B$ in the sense of Lemma~\ref{lem:Shadow} given by
    \begin{equation} \begin{split}
        T(v)&\coloneqq \begin{cases} v,\quad &v \in V; \\ w_a^\ell, \quad &v=w_b^\ell,\end{cases} \\
    T(\edge{v}{e}{w}) &\coloneqq \edge{T(v)}{e}{T(w)}.
    \end{split} \end{equation}
\end{proposition}

\begin{proof}
It is immediate that the formula for $T$ provides a homomorphism of Bratteli diagrams, and $T$ maps $V_{\mc C}^\ell$ surjectively to $V_\ell$ and $v_{\mc C}^0$ bijectively to $V_0$.

Next consider a vertex in $\mc C$.

If the vertex is $v^{\ell} \in V_{\ell}$, then $T(v^{\ell})=v^{\ell}$, and the edges starting at $v^{\ell}$ in $\mc C$ comprise:
\begin{itemize}
\item $\edge{v^{\ell}}e{v^{\ell+1}}$, for any $v^{\ell+1} \in V_{\ell+1}\setminus\{w_a^{\ell+1}\}$ and $\edge{v^{\ell}}e{v^{\ell+1}}\in E$; these map to $\edge{v^{\ell}}e{v^{\ell+1}}\in E$;
\item $\edge{v^{\ell}}e{w_b^{\ell+1}}$, for $e \in F(v^{\ell})$; these map to $\edge{v^{\ell}}e{w_a^{\ell+1}}\in E$; and
\item $\edge{v^{\ell}}e{w_a^{\ell+1}}$, for $e \in E(v^{\ell},w_a^{\ell+1})\setminus F(v^{\ell})$; these map to $\edge{v^{\ell}}e{w_a^{\ell+1}}\in E$.
\end{itemize}
From the first point, we see that $T$ maps edges from $v^{\ell}$ to $v^{\ell+1}$ in $\mc C$ bijectively to edges from $v^{\ell}$ to $v^{\ell+1}$ in $\mc B$ (for $v^{\ell+1} \neq w_a^{\ell+1}$).
From the second point, we see that $T$ maps edges from $v^{\ell}$ to $w_a^{\ell+1},w_b^{\ell+1}$ in $\mc C$ bijectively to edges from $v^{\ell}$ to $w_a^{\ell+1}$ in $\mc B$.

If the vertex is $w_b^{\ell}$, then $T(w_b^{\ell})=w_a^{\ell}$, and the edges starting at $w_b^{\ell}$ in $\mc C$ comprise:
\begin{itemize}
\item $\edge{w_b^{\ell}}e{v^{\ell+1}}$ for any $v^{\ell+1} \in V_{\ell+1}\setminus \{w_a^{\ell+1}\}$ and $\edge{w_a^{\ell}}e{v^{\ell+1}}\in E$; these map to $\edge{w_a^{\ell}}e{v^{\ell+1}}\in E$;
\item $\edge{w_b^{\ell}}e{w_b^{\ell+1}}$ in $E_{\mc C}$, for $e \in F(w_b^{\ell})$; these map to $\edge{w_a^{\ell}}e{w_a^{\ell+1}}\in E$; and
\item $\edge{w_b^{\ell}}e{w_a^{\ell+1}}$ in $E_{\mc C}$, for $e \in E(w_a^{\ell},w_a^{\ell+1})\setminus F(w_b^{\ell})$; these map to $\edge{w_a^{\ell}}e{w_a^{\ell+1}}\in E$.
\end{itemize}
From the first point, we see that $T$ maps edges from $w_b^{\ell}$ to $v^{\ell+1}$ in $\mc C$ bijectively to edges from $w_a^{\ell}$ to $v^{\ell+1}$ in $\mc B$ (for $v^{\ell+1} \neq w_a^{\ell+1}$).
From the second point, we see that $T$ maps edges from $v^{\ell}$ to $w_a^{\ell+1},w_b^{\ell+1}$ in $\mc C$ bijectively to edges from $v^{\ell}$ to $w_a^{\ell+1}$ in $\mc B$.

Altogether, $T$ satisfies the hypotheses of Lemma~\ref{lem:Shadow}.
\end{proof}

\section{A recipe for a path space homeomorphism}
\label{sec:TEBR}

\begin{definition}
\label{def:TEBR}
Let $\mc B$ be a Bratteli diagram and let $A,B\subseteq E_{0\dots\ell}$.
A \emph{recipe} is $(A,B,\rho)$ defined recursively, where $\rho=((w_1,z_1),\dots,(w_m,z_m),(A_1,B_1,\rho_1),\dots,(A_n,B_n,\rho_n))$ (for some $m,n$) and:
\begin{enumerate}
    \item $w_i,z_i \in E_{0,\dots,\ell+1}$ are two paths ending at the same vertex, for each $i=1,\dots,m$;
    \item $A_i,B_i \subseteq E_{0,\dots,\ell+1}$ and $(A_i,B_i,\rho_i)$ is a recipe, for each $i=1,\dots,n$; and
    \item $AE_{\ell,\ell+1} = \{w_1\}\sqcup \cdots \sqcup\{w_m\}\sqcup A_1\sqcup\cdots\sqcup A_n$ and $BE_{\ell,\ell+1} = \{z_1\}\sqcup \cdots \sqcup\{z_m\}\sqcup B_1\sqcup\cdots\sqcup B_n$.
\end{enumerate}
\end{definition}

The idea is that the data in a recipe $(A,B,\rho)$ aims to define a function $AX_{\mc B} \to BX_{\mc B}$, by doing tail-equivalence for $w_iX_{\mc B} \to z_iX_{\mc B}$ and recursion for $A_iX_{\mc B} \to B_iX_{\mc B}$; we formalize this in part \ref{it:TEBRfits} of the following definition.

\begin{definition}
\label{def:TEBRmore}
Let $\mc B$ be a Bratteli diagram, let $(A,B,\rho)$ be a recipe as in Definition~\ref{def:TEBR}, where $A,B\subseteq E_{0\dots\ell}$.
\begin{enumerate}
    \item \label{it:TEBRdescendant}
    The \emph{descendant recipes} are defined recursively to be $(A,B,\rho)$ as well as all descendant recipes of each $(A_i,B_i,\rho_i)$.
    \item \label{it:TEBRconstrained}
    For a sequence $(N_k)_{k=\ell}^\infty$, we recursively say $(A,B,\rho)$ is \emph{$(N_k)_{k=\ell}^\infty$-constrained} if all paths in $A$ agree on the first $N_\ell$ edges and all paths in $B$ agree on the first $N_\ell$ edges (i.e., there exists $w,z \in E_{0\dots N_{\ell}}$ such that $A\subseteq wE_{\dots\ell}$ and $B\subseteq zE_{\dots\ell}$), and all descendant recipes of $(A,B,\rho)$ satisfy the same condition.
    \item \label{it:TEBRfits} Recursively, we define a function $\varphi:AX_{\mc B}\to BX_{\mc B}$ to \emph{fit} $(A,B,\rho)$ if:
    \begin{enumerate}
        \item For $i=1,\dots,m$ and $w_ix \in w_iX_{\mc B}$, $\varphi(w_ix)=z_ix$, and
        \label{it:Fit.1}
        \item For $i=1,\dots,n$, $\varphi(A_iX_{\mc B})=B_iX_{\mc B}$ and $\varphi|_{A_iX_{\mc B}}^{B_iX_{\mc B}}$ fits $(A_i,B_i,\rho_i)$ (where $\varphi|_{A_iX_{\mc B}}^{B_iX_{\mc B}}$ denotes the corestriction to $B_iX_{\mc B}$ of the restriction to $A_iX_{\mc B}$ of $\varphi$).
        \label{it:Fit.2}
    \end{enumerate}
\end{enumerate}
\end{definition}

We now give two examples of recipes and functions which fit them, on the space of infinite words in $\{0,1\}$.

\begin{example}\label{eg:SomeTEBRs}
Let $\mc B$ be the Bratteli diagram
\begin{equation} \begin{tikzcd}[column sep=2mm,row sep=0.5mm]
    & \bullet \ar[ddddl,no head] \ar[ddddr,no head] & \\
    && \\
    && \\
    && \\
    0 \ar[dddd,no head] \ar[ddddrr,no head] && 1 \ar[ddddll,no head] \ar[dddd,no head] \\
    && \\
    && \\
    && \\
    0 && 1 \\
    & \vdots & \\
\end{tikzcd} \end{equation}
In this example we will write the vertex label where an edge lands (0 or 1) to refer to that edge; in this way, paths (both finite and infinite) identify with words in $\{0,1\}$.
Use $\bullet$ to refer to the path that consists only of the initial vertex (i.e., the word of length $0$).
In particular, the path space $X=X_{\mc B}$ identifies with $\{0,1\}^{\mathbb N}$.
    \begin{enumerate}
        \item (Odometer) \label{it:Odometer}
        Define the recipe $(\{\bullet\},\{\bullet\},\rho_{\mathrm{od}}=\rho_0)$ recursively by
\begin{equation} \Big(\{1^k\},\{0^k\},\rho_k\coloneqq \big((1^k0,0^k1),(\{1^{k+1}\},\{0^{k+1}\},\rho_{k+1})\big) \, \Big) \end{equation}
for $k =0,1,\dots$.
The unique function $\varphi_{\mathrm{od}}:X_{\mc B} \to X_{\mc B}$ that fits this recipe is the odometer on $\{0,1\}^\mathbb N$.
        \item \label{it:BitFlip}
Define $\hat 0\coloneqq 1$ and $\hat 1 \coloneqq 0$, and applying this character-by-character, we define $\hat w$ for any word $w$ in $\{0,1\}$.
Define a recipe $(\{\bullet\},\{\bullet\},\rho_{\mathrm{flip}}=\rho_\bullet)$ recursively by
\begin{equation} \Big(\{w\},\{\hat w\},\rho_w\coloneqq \big((\{w0\},\{\hat w1\},\rho_{w0}),(\{w1\},\{\hat w0\},\rho_{w1})\big) \, \Big) \end{equation}
for any finite word $w$ in $\{0,1\}$ (representing an element of $E_{0\dots\ell}$).
The unique function $\varphi_{\mathrm{flip}}:X_{\mc B} \to X_{\mc B}$ that fits this recipe is $w \mapsto \hat w$.

More generally, if $\mc B$ is a Bratteli diagram in which there is exactly one edge from each vertex at one level to each vertex at the next (so that words identify with sequences of vertices), and given a permutation $\sigma_\ell$ of the vertices $V_\ell$ at level $\ell$, one can produce a recipe similarly, so that the unique function that fits it sends a word $v_0 v_1 \cdots$ to the word $\sigma_0(v_0) \sigma_1(v_1) \cdots$.
    \end{enumerate}
\end{example}

In the above examples there is a unique function fitting the given recipe.
In general to arrange this, we ask that the constrained hypothesis (as in Definition~\ref{def:TEBRmore}\ref{it:TEBRconstrained} above) holds; this also ensures continuity of the fitting function.

\begin{lemma}\label{lem:TEBR}
    Let $\mc B$ be a Bratteli diagram and let $(A,B,\rho)$ be a recipe where $A,B\subseteq E_{0\dots\ell}$.
    Suppose that $\rho$ is $(N_\ell)_{\ell=1}^\infty$-constrained, for some sequence $N_\ell \to \infty$.
    Then there is a unique function $\varphi:AX_{\mc B} \to BX_{\mc B}$ that fits $(A,B,\rho)$.
    Moreover, $\varphi$ is a homeomorphism.
\end{lemma}

\begin{proof}
Assume the notation for $\rho$ as in Definition~\ref{def:TEBR}.
For simplicity, we can assume that $m=0$, as we now explain.
Say $m\geq 1$.
Let us enumerate the edges emanating from $w_1$ as $e_1,\dots,e_k$.
Then consider the recipe $(\{w_1\},\{z_1\},\rho')$ given by $\rho'\coloneqq ((w_1e_1,z_1e_1),\dots,(w_1e_k,z_1e_k))$.
Then $\varphi$ fits $(A,B,\rho)$ if and only if $\varphi$ fits the recipe 
\begin{equation}(A,B,((w_2,z_2),\dots,(w_m,z_m),(A_1,B_1,\rho_1),\dots,(A_n,B_n,\rho_n),(\{w_1\},\{z_1\},\rho')). \end{equation}
Iterating this, we can replace $(A,B,\rho)$ by a recipe in which $m=0$; moreover, by applying the same procedure iteratively, we can assume that the same holds for all descendants.

For simplicity, let us also assume that $\ell=0$ (so that $A,B \subseteq V_0$).
For consistency going forward, write $\rho_0=\rho$ and use the notation $\rho_0=((A_1^{(0)},B_1^{(0)},\rho_1^{(0)}),\dots,(A_{n_0}^{(0)},B_{n_0}^{(0)},\rho_{n_0}^{(0)}))$.

Consider a point $x=\halfedge{v_0}{e_1}\halfedge{v_1}{e_2}\cdots \in AX_{\mc B}$.
Since $v_0E_{0\dots1}=A_1^{(0)}\sqcup\cdots \sqcup A_{n_0}^{(0)}$, there is a unique $i_0\in\{1,\dots,n_0\}$ such that $\edge{v_0}{e_1}{v_1} \in A_{i_0}^{(0)}$.
Note then that for $\varphi$ to fit $(A,B,\rho)$, it must be the case that $\varphi(A_{i_0}^{(0)})\subseteq B_{i_0}^{(0)}$, and in particular,
\begin{equation} \varphi(x) \in B_{i_0}^{(0)}X_{\mc B}. \end{equation}

Set $\rho_1\coloneqq \rho_{i_0}^{(0)}$, and expand it as 
$\rho_1=((A_1^{(1)},B_1^{(1)},\rho_1^{(1)}),\dots,(A_{n_1}^{(1)},B_{n_1}^{(1)},\rho_{n_1}^{(1)})$.
Then, similarly, there is a unique $i_1$ such that $\halfedge{v_0}{e_1}\edge{v_1}{e_2}{v_2} \in A_{i_1}^{(1)}$, and then $\varphi$ will need to satisfy
\begin{equation} \varphi(x) \in B_{i_1}^{(0)}X_{\mc B}. \end{equation}
Set $\rho_2\coloneqq \rho_{i_1}^{(1)}$ and continue.

This yields sequences $(A_{i_\ell}^{(\ell)})_{\ell=1}^\infty$ and $(B_{i_\ell}^{(\ell)})_{\ell=1}^\infty$, which in turn yield decreasing sequences of clopen sets
\begin{equation} \begin{split}
A_{i_0}^{(0)}X_{\mc B} \supseteq A_{i_1}^{(1)}X_{\mc B} \supseteq \cdots, \\
B_{i_0}^{(0)}X_{\mc B} \supseteq B_{i_1}^{(1)}X_{\mc B} \supseteq \cdots, \\
\end{split} \end{equation}
such that $x \in A_{i_\ell}^{(\ell)}X_{\mc B}$ for all $\ell$, and therefore $\varphi(x)$ must be in $B_{i_\ell}^{(\ell)}X_{\mc B}$ for all $\ell$.
The condition that $\rho$ is $(N_\ell)_{\ell=1}^\infty$-constrained implies that each $B_{i_\ell}^{(\ell)}$ is contained in a cylinder defined by a path of length $N_\ell$; since $N_\ell \to \infty$, it follows (using compactness) that $\bigcap_\ell B_{i_\ell}^{(\ell)}$ consists of a single point $y$, and we therefore (must) define $\varphi(x)\coloneqq y$.

This shows that $\varphi$ exists and is unique.
We see that $\varphi$ maps each cylinder given by a path of length $\ell$ into a cylinder of path of length $N_\ell$, so again since $N_\ell \to \infty$, this implies that $\varphi$ is continuous.
As we may apply the same with the roles of $A$ and $B$ reversed, we obtain a continuous inverse to $\varphi$, showing that $\varphi$ is a homeomorphism.
\end{proof}

\begin{lemma}\label{lem:TEBRshadow}
    Let $\mc B$ be a Bratteli diagram and let $(A,B,\rho)$ be a recipe which is $(N_\ell)_{\ell=1}^\infty$-constrained for some sequence $N_\ell \to \infty$, and let $\varphi:AX_{\mc B} \to BX_{\mc B}$ be the homeomorphism that fits $\rho$.
    Suppose that the equivalence relation $\mc R$ generated by $\mc R_{\mc B,\mathrm{tail}}$ and $\varphi$ is \'etale.
    
    Let $\mc B'$ be another Bratteli diagram and suppose that there is a shadow Bratteli diagram  homomorphism $T:\mc B \to \mc B'$, in the sense of Lemma~\ref{lem:Shadow}.
    Suppose that, for every descendant $(A',B',\rho')=((w_1,z_1),\dots,(w_m,z_m),(A_1,B_1,\rho_1),\dots,(A_n,B_n,\rho_n))$ of $\rho$, and every $i=1,\dots,n$, the vertices that paths in $T(A_i)$ end at are the same as the vertices that paths in $T(B_i)$ end at, counted with multiplicity.
    Then the shadow homeomorphism $s_X:X_{\mc B} \to X_{\mc B'}$ induces a surjective homomorphism $S(\mc R) \to S(\mc R_{\mc B',\mathrm{tail}})$.
\end{lemma}

\begin{proof}
    We want to check the hypothesis of Proposition~\ref{prop:AbstractHomomorphism}, where $\mc G\coloneqq \mc R$, $\mc H \coloneqq \mc R_{\mc B',\mathrm{tail}}$, and $\theta\coloneqq s_X$.
    Since $s_X$ induces a continuous injective groupoid homomorphism $s_{\mathrm{tail}}:\mc R_{\mc B,\mathrm{tail}} \to \mc R_{\mc B',\mathrm{tail}}$, the hypothesis holds for any bisection coming from $\mc R_{\mc B,\mathrm{tail}}$.
    Since $\mc R$ is generated by $\mc R_{\mc B,\mathrm{tail}}$ and $\varphi$, it suffices to check the hypothesis for bisections coming from $\varphi$; that is, bisections of the form $U=\{(x,\varphi(x)): x \in Z\}$ for some clopen set $Z\subseteq A$.
    
    Consider a descendant $(A',B',\rho')$ where $\rho'=((w_1,z_1),\dots,(w_m,z_m),(A_1,B_1,\rho_1),\dots,(A_n,B_n,\rho_n))$.
    If $Z=w_iX$ for some $i=1,\dots,m$, then $U$ comes from tail-equivalence, so the hypothesis holds.
    On the other hand, if $Z=A_iX$ for some $i=1,\dots,n$, then by the hypothesis, $s_X(A_iX)$ and $s_X(B_iX)$ are equivalent via tail equivalence.

    Any clopen set can be written as a finite disjoint union of cylinders.
    Moreover, for $w \in E_{0\dots k}$, if we take $\ell$ such that $N_\ell \geq k$, it follows that the cylinder $wX_{\mc B}$ can be decomposed as a union of clopen sets coming from descendants at level at most $\ell$ (that is, sets $w_iX$ or $A_iX$ where the descendant is $\rho'$ as in the previous paragraph).
    Consequently, we see that the hypothesis of Proposition~\ref{prop:AbstractHomomorphism} holds for any clopen set, and therefore $s_X$ does induce a surjective homomorphism $S(\mc R) \to S(\mc R_{\mc B',\mathrm{tail}})$.
\end{proof}

\begin{remark}\label{rmk:ReversingTEBR}
Given a recipe $(A,B,\rho)$ with $\rho=((w_1,z_1),\dots,(w_m,z_m),(A_1,B_1,\rho_1),\dots,(A_n,B_n,\rho_n))$ (for some $m,n$), we can reverse it to produce a recipe $(B,A,\rho_{\mathrm r})$, with 
\begin{equation} \rho_{\mathrm r}=((z_1,w_1),\dots,(z_m,w_m),(B_1,A_1,\rho_{1,\mathrm r}),\dots,(B_n,A_n,\rho_{n,\mathrm r})), \end{equation}
where we recursively reverse each $(A_i,B_i,\rho_i)$ to produce $(B_i,A_i,\rho_{i,\mathrm r})$.
If $\varphi:AX_{\mc B} \to BX_{\mc B}$ is a homeomorphism fitting $(A,B,\rho)$, then $\varphi^{-1}$ fits $(B,A,\rho_{\mathrm r})$.
\end{remark}

\section{Obtaining an \'etale equivalence relation from a homeomorphism and tail-equivalence}
\label{sec:TEBRetale}

We now consider the matter of when a homeomorphism fitting a recipe is compatible with tail-equivalence, in the sense that the combined equivalence relation is \'etale.
For this we will make use of Proposition~\ref{prop:EtaleCombinedRelation}.
We start by providing some examples to show that this does not happen automatically; this will motivate the introduction of conditions that turn out to be sufficient.

\begin{example}\label{eg:NonEtaleFlip1}
Let $\mc B$ be the Bratteli diagram
\begin{equation} \begin{tikzcd}[column sep=2mm,row sep=0.5mm]
    && \bullet \ar[ddddll,no head] \ar[dddd,no head] \ar[ddddrr,no head] && \\
    &&&& \\
    &&&& \\
    &&&& \\
    a \ar[dddd,no head] \ar[ddddrr,no head]\ar[ddddrrrr,no head] && b \ar[ddddll,no head] \ar[dddd,no head] \ar[ddddrr,no head] && c \ar[ddddllll,no head] \ar[ddddll, no head] \ar[dddd,no head] \\
    &&&& \\
    &&&& \\
    &&&& \\
    a&&b&&c \\
    && \vdots && \\
\end{tikzcd} \end{equation}
and as in Example~\ref{eg:SomeTEBRs}, identify paths in $\mc B$ with words in $\{a,b,c\}$.
Let $(\{\bullet\},\{\bullet\},\rho)$ be the recipe as described at the end of Example~\ref{eg:SomeTEBRs}\ref{it:BitFlip}, such that the fitting homeomorphism $\varphi:X_{\mc B} \to X_{\mc B}$ is the one taking a word $x_1x_2\cdots$ in $\{a,b,c\}^{\mathbb N}$ to $\sigma(x_1)\sigma(x_2)\cdots$ where 
\begin{equation}\label{eq:EgSigma} \sigma(a)=a,\quad \sigma(b)=c,\quad \text{and}\quad \sigma(c)=b. \end{equation}

Then $\varphi(\overline{a})=\overline{a}$, but there is no other point (let alone a neighbourhood of points) fixed by $\varphi$.
If there were an \'etale topology for $\mc R_\varphi$, then both $\{(x,\varphi(x)): x \in X_{\mc B}\}$ and $\{(x,x): x\in X_{\mc B}\}$ would be open subsets, but their intersection is $\{(\overline{a},\overline{a})\}$, which cannot be open.
\end{example}

\begin{example}\label{eg:NonEtaleFlip2}
The next example is a variant on the previous, in which $\mc R_\varphi$ is an \'etale equivalence relation, but its combination with tail-equivalence no longer is.
Let $\mc B$ be the Bratteli diagram
\begin{equation} \begin{tikzcd}[column sep=2mm,row sep=0.5mm]
    && \bullet \ar[dddd,no head] \ar[ddddrr,no head] && \\
    &&&& \\
    &&&& \\
    &&&& \\
    && b \ar[ddddll,no head] \ar[dddd,no head]\ar[ddddrr,no head] && c \ar[ddddllll,no head] \ar[ddddll,no head] \ar[dddd,no head] \\
    &&&& \\
    &&&& \\
    &&&& \\
    a \ar[dddd,no head] \ar[ddddrr,no head]\ar[ddddrrrr,no head] && b \ar[ddddll,no head] \ar[dddd,no head] \ar[ddddrr,no head] && c \ar[ddddllll,no head] \ar[ddddll, no head] \ar[dddd,no head] \\
    &&&& \\
    &&&& \\
    &&&& \\
    a&&b&&c \\
    && \vdots && \\
\end{tikzcd} \end{equation}
that is, we simply removed one vertex labelled $a$ from the previous example.
We may, as in the previous example, create a recipe whose fitting homeomorphism $\varphi$ is the restriction of the previous one to this tail space: that is, for a word $x_1x_2\cdots$ with $x_1 \in \{b,c\}$ and $x_i\in\{a,b,c\}$ for $i > 1$, $\varphi(x_1x_2\cdots) = \sigma(x_1)\sigma(x_2)\cdots$ where $\sigma$ is given by \eqref{eq:EgSigma}.
In this case, $\varphi$ has no fixed points and (as before) $\varphi$ has order two -- in other words, $\varphi$ comes from a free action of $\mathbb Z_2$.
Using this, one can show that $\mc R_\varphi$ has an \'etale topology.

However, $V\coloneqq \{(x,\varphi(x)):x\in X\}$ is an open set in $\mc R_\varphi$, whereas $U_{bc,ca}$ (see \eqref{eq:TailEqBasis}) is an open set in $\mc R_{\mc B,\mathrm{tail}}$.
The intersection of these two sets contains only the point $(b\overline{a},c\overline{a})$, so that $U\cap V$ could not be an open set in an \'etale topology.
This shows that there is no \'etale topology for the equivalence relation generated by $\varphi$ and tail-equivalence.
\end{example}

\begin{example}\label{eg:NonEtaleFlipPlusOdometer}
Let $\mc B$ be the Bratteli diagram from Example \ref{eg:SomeTEBRs}, and again as in that example, identify paths with words in $\{0,1\}$ and $X=\{0,1\}^{\mathbb N}$.
Obtain a recipe $(\{0\},\{0\},\rho_{0,\mathrm{od}})$ exactly as in Example \ref{eg:SomeTEBRs}\ref{it:Odometer} but prefixing all paths with $0$, and a recipe $(\{1\},\{1\},\rho_{1,\mathrm{flip}})$ exactly as in Example \ref{eg:SomeTEBRs}\ref{it:BitFlip} rooted at the vertex $1$ on the first level.
Define a recipe $\left(\{\bullet\},\{\bullet\},\rho=\big((\{0\},\{0\},\rho_{0,\mathrm{od}}),(\{1\},\{1\},\rho_{1,\mathrm{flip}})\big)\,\right)$.
Let $\varphi_{\mathrm{od}},\varphi_{\mathrm{flip}}: X\to X$ be from Example~\ref{eg:SomeTEBRs}.
Then the homeomorphism $\varphi:X \to X$ that fits this recipe is given by
\begin{equation} \varphi(0x)= 0\varphi_{\mathrm{od}}(x),\quad \varphi(1x)=1\varphi_{\mathrm{flip}}(x), \end{equation}
for all $x \in X$.
Since $\varphi_{\mathrm{od}}$ and $\varphi_{\mathrm{flip}}$ each generate an \'etale equivalence relation, we have that $\mc R_\varphi$ is also \'etale.
(One can even check that each of the recipes in Example~\ref{eg:SomeTEBRs}, when combined with tail-equivalence, produce an \'etale equivalence relation.)

However, let $\mc R$ be the equivalence relation generated by $\varphi$ and tail-equivalence, and consider the sets
\begin{equation}\begin{split}
U&\coloneqq \{(0x,0\varphi_{\mathrm{od}}(x)): x \in X\}, \\
V&\coloneqq \{(0x,0\varphi_{\mathrm{flip}}(x)): x \in X\}.
\end{split} \end{equation}
We have that $U$ is just $\{(x,\varphi(x)):x \in 0X\}$, a clopen bisection in $\mc R_\varphi$, whereas
\begin{equation} V = U_{0,1} \{(x,\varphi(x)):x \in 1X\} U_{1,0}, \end{equation}
where the first and last are clopen bisections in $\mc R_{\mc B,\mathrm{tail}}$ (from \eqref{eq:TailEqBasis}) and the middle is a clopen bisection in $\mc R_\varphi$.
Hence $V$ would be open in $\mc R$.
The intersection of $U$ and $V$ consists of the single point $(0\overline{1},0\overline{0})$.
As before, we conclude that there is no \'etale topology of $\mc R$ which contains the open sets from both $\mc R_\varphi$ and $\mc R_{\mc B,\mathrm{tail}}$.
\end{example}

A homeomorphism fitting a recipe will map many points to tail-equivalent points, but may have some ``critical points'' which are mapped to points that are not tail-equivalent.
The following gives a combinatorial condition for the set of ``critical points'' (in this sense) to be disjoint from its image, even up to tail-equivalence.
This will rule out the sort of behaviour in Examples \ref{eg:NonEtaleFlip1} and \ref{eg:NonEtaleFlip2}.

\begin{lemma}
\label{lem:DisjointHypothesis}
    Let $\mc B$ be a Bratteli diagram and let $(A,B,\rho)$ be a recipe which is $(N_\ell)_{\ell=1}^\infty$-constrained for some sequence $N_\ell \to \infty$, and let $\varphi:AX_{\mc B} \to BX_{\mc B}$ be the homeomorphism that fits $\rho$.

    Suppose that for any descendant recipes $(A',B',\rho')$ and $(A'',B'',\rho'')$ at the same level, the set of vertices where paths in $A'$ end is disjoint from the set of vertices where paths in $B''$ end.

    Then for $x \in AX_{\mc B},y \in BX_{\mc B}$, if $\varphi(x)$ and $y$ are tail-equivalent, then either:
\begin{enumerate}
\item \label{it:DisjointHypothesis.1}
$x$ is tail-equivalent to $\varphi(x)$, or
\item \label{it:DisjointHypothesis.2}
$y$ is tail-equivalent to $\varphi(y)$.
\end{enumerate}
Moreover, if $x$ is tail-equivalent to $\varphi(x)$ then there is a clopen set $U$ containing $x$ such that $\{(x,\varphi(x)):x \in U\}$ is a clopen bisection in $\mc R_{\mc B,\mathrm{tail}}$.
\end{lemma}

\begin{proof}
    We prove the last statement first, as it helps explain what is going on in the first statement.
    If $x$ and $\varphi(x)$ are tail-equivalent, then at some level the prefix of $x$ and of $\varphi(x)$ must end at the same vertex.
    By the disjointedness condition, it is impossible at this level for there to be a descendant recipe $(A',B',\rho')$ such that $x \in A'X_{\mc B}$.
    Hence instead we must have that at this (or an earlier) level, some $(w_i,z_i)$ in the recipe are prefixes of $x,\varphi(x)$ respectively.
    Then $U\coloneqq w_iX_{\mc B}$ satisfies
\begin{equation} \{(x,\varphi(x)): x\in U\} = U_{w_i,z_i}, \end{equation}
where $U_{w_i,z_i}$ is the clopen bisection in $\mc R_{\mc B,\mathrm{tail}}$ defined in \eqref{eq:TailEqBasis}.

    For the first statement, suppose, for a contradiction, that neither \ref{it:DisjointHypothesis.1} nor \ref{it:DisjointHypothesis.2} hold.
    Let $\ell$ be such that $\varphi(x)$ and $y$ agree after level $\ell$.
    In particular, they go to the same vertex at level $\ell$.

    Since \ref{it:DisjointHypothesis.2} does not hold, we see that there must be a descendant recipe $(A',B',\rho')$ at level $\ell$, such that $y \in A'X_{\mc B}$.
    Likewise, since \ref{it:DisjointHypothesis.1} does not hold, there is a descendant recipe $(A'',B'',\rho'')$ at level $\ell$ such that $x \in A''X_{\mc B}$ and $\varphi(x) \in B''X_{\mc B}$.
    But since $\varphi(x)$ and $y$ go to the same vertex at level $\ell$, then there must be a vertex that a path in both $A'$ and $B''$ end at.
    This contradicts the hypothesis.
\end{proof}

We now introduce a symmetry-type condition that will rule out the sort of behaviour in Example~\ref{eg:NonEtaleFlipPlusOdometer}.

\begin{definition}
\label{def:Model}
Let $\mc B$ be a Bratteli diagram, let 
\begin{equation} \Big(A,B,\rho=\big((w_1,z_1),\dots,(w_m,z_m),(A_1,B_1,\rho_1),\dots,(A_n,B_n,\rho_n)\big)\,\Big) \end{equation}
be a recipe with $A,B \subseteq E_{0\dots \ell}$, let $0\leq k_0\leq \ell$, and let
\begin{equation}
 \begin{split}
&\Big(A_{\mathrm{mod}},B_{\mathrm{mod}},
\\ &\qquad
\rho_{\mathrm{mod}}=\big((w^{\mathrm{mod}}_{1},z^{\mathrm{mod}}_{1}),\dots,(w^{\mathrm{mod}}_{m'},z^{\mathrm{mod}}_{m'}),
(A^{\mathrm{mod}}_{1},B^{\mathrm{mod}}_{1},\rho^{\mathrm{mod}}_{1}),\dots,(A^{\mathrm{mod}}_{n'},B^{\mathrm{mod}}_{n'},\rho^{\mathrm{mod}}_{n'})\big)\,\Big) 
\end{split} \end{equation}
be a recipe for the Bratteli diagram $\mathrm{Trunc}_{k_0}(\mc B)$, with $A,B\subseteq E_{k_0\dots \ell}$.
Recursively we define what it means for $(A_{\mathrm{mod}},B_{\mathrm{mod}},\rho_{\mathrm{mod}})$ to \emph{model} $(A,B,\rho)$ as follows
\begin{enumerate}
\item $\mathrm{Trunc}_{k_0}|_A$ is a bijection $A \to A_{\mathrm{mod}}$;
\item $\mathrm{Trunc}_{k_0}|_B$ is a bijection $B \to B_{\mathrm{mod}}$;
\item $m=m'$, $\mathrm{Trunc}_{k_0}(w_i)=w^{\mathrm{mod}}_i$, and $\mathrm{Trunc}_{k_0}(z_i)=z^{\mathrm{mod}}_i$, for all $i=1,\dots,m$; and
\item $n=n'$ and $(A^{\mathrm{mod}}_{i},B^{\mathrm{mod}}_{i},\rho^{\mathrm{mod}}_{i})$ models $(A_i,B_i,\rho_i)$ for $i=1,\dots,n$.
\end{enumerate}
A recipe $(A,B,\rho=((w_1,z_1),\dots,(w_m,z_m)),(A_1,B_1,\rho_1),\dots,(A_n,B_n,\rho_n))$ is defined recursively to be \emph{highly symmetric} if there is some recipe $(A_{\mathrm{mod}},B_{\mathrm{mod}},\rho_{\mathrm{mod}})$ for some truncated Bratteli diagram $\mathrm{Trunc}_{k_0}(\mc B)$ which simultaneously models all the $(A_i,B_i,\rho_i)$, and $(A_{\mathrm{mod}},B_{\mathrm{mod}},\rho_{\mathrm{mod}})$ is highly symmetric.
Equivalently, for every $\ell$, there is a recipe $(A_{\mathrm{mod},\ell},B_{\mathrm{mod},\ell},\rho_{\mathrm{mod},\ell})$ for some truncated Bratteli diagram $\mathrm{Trunc}_{k_\ell}(\mc B)$, which is a model for every descendant at level $\ell$.
\end{definition}

\begin{lemma}
\label{lem:ModelHomeo}
Let $\mc B$ be a Bratteli diagram, let $(A,B,\rho)$ be a recipe, let $k_0 \geq 1$ and let $(A_{\mathrm{mod}},B_{\mathrm{mod}},\rho_{\mathrm{mod}})$ be a recipe for $\mathrm{Trunc}_{k_0}(\mc B)$ which models $(A,B,\rho)$.
If $\varphi:AX_{\mc B} \to BX_{\mc B}$ is a homeomorphism that fits $(A,B,\rho)$, then there is a homeomorphism 
\begin{equation} \varphi_{\mathrm{mod}}:A_{\mathrm{mod}}X_{\mathrm{Trunc}_{k_0}(\mc B)} \to B_{\mathrm{mod}}X_{\mathrm{Trunc}_{k_0}(\mc B)} \end{equation}
such that $\varphi_{\mathrm{mod}} \circ \mathrm{Trunc}_{k_0} = \mathrm{Trunc}_{k_0} \circ \varphi$ on $AX_{\mc B}$.
\end{lemma}

\begin{proof}
Note that since $\mathrm{Trunc}_{k_0}$ takes $A$ bijectively to $A_{\mathrm{mod}}$, it follows that $\mathrm{Trunc}_{k_0}|_{AX_{\mc B}}$ is a homeomorphism
\begin{equation} AX_{\mc B} \to A_{\mathrm{mod}}X_{\mathrm{Trunc}_{k_0}(\mc B)} \end{equation}
Let $\Psi:A_{\mathrm{mod}}X_{\mathrm{Trunc}_{k_0}(\mc B)} \to AX_{\mc B}$ denote the inverse, and using this, define
\begin{equation} \varphi_{\mathrm{mod}} \coloneqq \mathrm{Trunc}_{k_0} \circ \varphi \circ \Psi.\qedhere \end{equation}
\end{proof}

\begin{lemma}
\label{lem:HighlySymmHypothesis}
Let $\mc B$ be a Bratteli diagram, let $(A,B,\rho)$ be a recipe which is highly symmetric.
    Suppose that $\rho$ is $(N_\ell)_{\ell=1}^\infty$-constrained, for some sequence $N_\ell \to \infty$, and let $\varphi:AX_{\mc B} \to BX_{\mc B}$ be the (unique) homeomorphism which fits $(A,B,\rho)$.
    Then for $x \in AX_{\mc B},y \in BX_{\mc B}$, if $x$ is tail-equivalent to $y$ then either:
\begin{enumerate}
\item \label{it:HighlySymmHypothesis.1}
$x$ is tail-equivalent to $\varphi(x)$,
\item \label{it:HighlySymmHypothesis.2}
$y$ is tail-equivalent to $\varphi(y)$, or
\item \label{it:HighlySymmHypothesis.3}
$\varphi(x)$ is tail-equivalent to $\varphi(y)$ and moreover there is a clopen bisection $V$ in $\mc R_{\mc B,\mathrm{tail}}$ containing $(x,y)$ such that $\{(\varphi(x'),\varphi(y')): (x',y') \in V\}$ is a clopen bisection in $\mc R_{\mc B,\mathrm{tail}}$.
\end{enumerate}
\end{lemma}

\begin{proof}
Assume that neither \ref{it:HighlySymmHypothesis.1} nor \ref{it:HighlySymmHypothesis.2} hold.
Using the hypothesis that $N_\ell \to \infty$ and that $x,y$ are tail-equivalent, let $\ell\in\mathbb N$ be such that $\mathrm{Trunc}_{N_\ell}(x)=\mathrm{Trunc}_{N_\ell}(y)$.
Since \ref{it:HighlySymmHypothesis.1} does not hold, it follows that there is a descendant of $(A,B,\rho)$ at every level containing the start of $x$ (and likewise for $y$, since \ref{it:HighlySymmHypothesis.2} does not hold); therefore let $(A_x,B_x,\rho_x)$, $(A_y,B_y,\rho_y)$ be descendants at level $\ell$ such that the start of $x$ is in $A_x$ and the start of $y$ is in $A_y$.
By the highly-symmetric hypothesis, let $k_0 \in \mathbb N$ and let $(A_{\mathrm{mod}},B_{\mathrm{mod}},\rho_{\mathrm{mod}})$ be a recipe for $\mathrm{Trunc}_{k_0}(\mc B)$ which models $(A_x,B_x,\rho_x)$ and $(A_y,B_y,\rho_y)$ (and indeed, it models all descendants at level $\ell$).

By the constrained hypothesis, we see that $\mathrm{Trunc}_{N_\ell}$ is injective when restricted to each of the sets $A_x,B_x,A_y,B_y$; we may therefore assume that $k_0 \geq N_\ell$ (if we replace $(A_{\mathrm{mod}},B_{\mathrm{mod}},\rho_{\mathrm{mod}})$ by its truncation to $N_\ell$, then it will still be a model for $(A_x,B_x,\rho_x)$ and $(A_y,B_y,\rho_y)$).
We also know that $(A_{\mathrm{mod}},B_{\mathrm{mod}},\rho_{\mathrm{mod}})$ is $(N_\ell)_{\ell\geq k_0}$-constrained, so there is a unique homeomorphism $\varphi_{\mathrm{mod}}$ fitting it.
By Lemma~\ref{lem:ModelHomeo} (and uniqueness of $\varphi_{\mathrm{mod}}$), we have
\begin{equation} \varphi_{\mathrm{mod}}\circ\mathrm{Trunc}_{k_0} = \mathrm{Trunc}_{k_0}\circ\varphi \end{equation}
on both $A_xX_{\mc B}$ and $A_yX_{\mc B}$.
Therefore,
\begin{equation} \begin{split}
\mathrm{Trunc}_{k_0}(\varphi(x)) &= \varphi_{\mathrm{mod}}(\mathrm{Trunc}_{k_0}(x)) \\
&= \varphi_{\mathrm{mod}}(\mathrm{Trunc}_{k_0}(y)) \\
&= \mathrm{Trunc}_{k_0}(\varphi(y)),
\end{split} \end{equation}
showing that $\varphi(x),\varphi(y)$ are tail-equivalent.
Moreover, if $w_x,w_y,z_x,z_y$ are the prefixes of $x,y,\varphi(x),\varphi(y)$ of length $k_0$, then with $U_{w_x,w_y}$ as defined in \eqref{eq:TailEqBasis}, we have
\begin{equation} \{(\varphi(w),\varphi(z)): (w,z) \in U_{w_x,w_y}\} = U_{z_x,z_y}, \end{equation}
so \ref{it:HighlySymmHypothesis.3} holds with $V\coloneqq U_{w_x,w_y}$.
\end{proof}

\begin{lemma}\label{lem:TEBRcompatibility}
    Let $\mc B$ be a Bratteli diagram and let $(A,B,\rho)$ be a highly symmetric recipe which is $(N_\ell)_{\ell=1}^\infty$-constrained for some sequence $N_\ell \to \infty$, and let $\varphi:AX_{\mc B} \to BX_{\mc B}$ be the homeomorphism that fits $\rho$.
    Suppose that for any descendant recipe $(A',B',\rho')$, the set of vertices that $A'$ ends at is disjoint from the set of vertices that $B'$ ends at.
    Let $\mc R$ be the equivalence relation generated by $\mc R_{\mc B,\mathrm{tail}}$ and $\varphi$, and let $\Gamma$ be the set of $\mc R$-bisections as given by \eqref{eq:CombinedEquivRelation}.
    For $\gamma \in \Gamma$ and $(x,y) \in \gamma$, there is a bisection $\gamma_0 \in \Gamma$ contained in $\gamma$ such that $(x,y) \in \gamma_0$ and such that either:
    \begin{enumerate}
        \item \label{it:TEBRcompatibility.1} $\gamma_0$ is a clopen bisection of $\mc R_{\mc B,\mathrm{tail}}$, or
	\item \label{it:TEBRcompatibility.2} $\gamma_0$ factors as $\gamma_1\gamma_2\gamma_3$ where $\gamma_1,\gamma_3$ are clopen bisections of $\mc R_{\mc B,\mathrm{tail}}$ and $\gamma_2$ is a clopen bisection of $\mc R_\varphi$.
    \end{enumerate}
Moreover, if $(x,y) \in \mc R_{\mc B,\mathrm{tail}}$ then (by possibly shrinking $\gamma_0$), \ref{it:TEBRcompatibility.1} must occur.
\end{lemma}

\begin{proof}
We prove the main statement by induction.
Since every $\gamma \in \Gamma$ is a finite product of clopen bisections from $\mc R_{\mc B,\mathrm{tail}}$ and $\mc R_\varphi$, by induction we may assume that
\begin{equation} \gamma = \gamma'\delta \end{equation}
where:
\begin{itemize}
\item $\gamma'$ is either a clopen bisection in $\mc R_{\mc B,\mathrm{tail}}$ or it factors as in \ref{it:TEBRcompatibility.2},
\item $\delta$ is a clopen bisection from either $\mc R_{\mc B,\mathrm{tail}}$ or $\mc R_\varphi$, and
\item $(x,y)$ factors as $(x,z)(z,y)$ where $(x,z) \in \gamma'$ and $(z,y)\in \delta$.
\end{itemize}

If $\gamma'$ is a clopen bisection in $\mc R_{\mc B,\mathrm{tail}}$ then it is clear that the statement holds for $(x,y) \in \gamma=\gamma'\delta$.
Likewise, if $\delta \in \mc R_{\mc B,\mathrm{tail}}$ and $\gamma'$ factors as in \ref{it:TEBRcompatibility.2}, then so does $\gamma'\delta$.

We are left to consider the interesting case: when $\gamma'=\gamma_1\gamma_2\gamma_3$ where $\gamma_1,\gamma_3$ are clopen bisections in $\mc R_{\mc B,\mathrm{tail}}$, and $\gamma_2 ,\delta$ are clopen bisections in $\mc R_\varphi$.
Let us further factorize $(x,z)$ as the product of $(x,z_1)\in \gamma_1$, $(z_1,z_2)\in \gamma_2$, and $(z_2,z_3) \in \gamma_3$.
If $z_1,z_2$ are tail-equivalent then by the last statement of Lemma~\ref{lem:DisjointHypothesis}, we can assume that $\gamma_2 \subseteq \mc R_{\mc B,\mathrm{tail}}$, so that $\gamma'\subseteq \mc R_{\mc B,\mathrm{tail}}$ (a case we already handled).
Therefore we will assume that $z_1,z_2$ are not tail-equivalent.
Likewise, we may assume that $z_3,y$ are not tail-equivalent.

Since $(z_1,z_2) \in \mc R_{\varphi}$, by \eqref{eq:Rphidesc}, we either have $z_1=z_2, z_2=\varphi(z_1),$ or $z_1=\varphi(z_2)$ (but the first cannot occur since we assume $z_1,z_2$ are not tail-equivalent).
Likewise, either $y=\varphi(z_3)$ or $z_3=\varphi(y)$.
Altogether, this gives us four cases to consider.

If $z_2=\varphi(z_1)$ and $y=\varphi(z_3)$, then by Lemma~\ref{lem:DisjointHypothesis} (applied using $z_1$ for $x$ and $z_3$ for $y$), either $z_1,z_2$ are tail-equivalent or $y,z_3$ are tail-equivalent.
By assumption, neither of these can hold, so this case does not occur.

If $z_1=\varphi(z_2)$ and $y=\varphi(z_3)$ then by Lemma~\ref{lem:HighlySymmHypothesis} (applied using $z_2$ for $x$ and $z_3$ for $y$), and assuming that either $z_1,z_2$ nor $y,z_3$ are tail-equivalent, then $y,z_3$ are tail-equivalent and, by possibly shrinking $\gamma_3$, we have that $\gamma_2\gamma_3\delta$ is a clopen bisection in $\mc R_{\mc B,\mathrm{tail}}$.
In that case, $\gamma'=\gamma_1(\gamma_2\gamma_3\delta)$ is a clopen bisection in $\mc R_{\mc B,\mathrm{tail}}$.

If $z_2=\varphi(z_1)$ and $z_3=\varphi(y)$ then we can apply the same argument as above with $\varphi^{-1}$ in place of $\varphi$; note that $\varphi^{-1}$ can be constructed by reversing the recipe for $\varphi$, and this reversed recipe satisfies the same hypotheses.
Hence in this case, we can also conclude that $\gamma'$ contains a clopen bisection in $\mc R_{\mc B,\mathrm{tail}}$ which contains $(x,y)$.

Finally, if $z_1=\varphi(z_2)$ and $z_3=\varphi(y)$ then by applying Lemma~\ref{lem:DisjointHypothesis} (with $y$ for $x$ and $z_2$ for $y$), either $y,z_3$ or $z_1,z_2$ are tail-equivalent, but again by assumption neither can occur.

For the final statement of the lemma, suppose that $(x,y) \in \mc R_{\mc B,\mathrm{tail}}$ and \ref{it:TEBRcompatibility.2} holds.
In this case, factor $(x,y)$ as the product of $(x,z_1) \in \gamma_1$, $(z_1,z_2)\in\gamma_2$, and $(z_2,y)\in \gamma_3$.
Then since $(x,y),(x,z_1),(z_2,y)\in\mc R_{\mc B,\mathrm{tail}}$, it follows that $(z_1,z_2)\in\mc R_{\mc B,\mathrm{tail}}$ as well.
Hence by the final statement of Lemma~\ref{lem:DisjointHypothesis}, by possibly shrinking $\gamma_2$, we may assume it is a clopen bisection in $\mc R_{\mc B,\mathrm{tail}}$.
Hence so is $\gamma = \gamma_1\gamma_2\gamma_3$.
\end{proof}

\begin{corollary}
\label{cor:EtaleUnderConditions}
    Let $\mc B$ be a Bratteli diagram and let $(A,B,\rho)$ be a highly symmetric recipe which is $(N_\ell)_{\ell=1}^\infty$-constrained for some sequence $N_\ell \to \infty$, and let $\varphi:AX_{\mc B} \to BX_{\mc B}$ be the homeomorphism that fits $\rho$.
    Suppose that for any descendant recipe $(A',B',\rho')$, the set of vertices that $A'$ ends at is disjoint from the set of vertices that $B'$ ends at.
    Let $\mc R$ be the equivalence relation generated by $\mc R_{\mc B,\mathrm{tail}}$ and $\varphi$, and let $\Gamma$ be the set of $\mc R$-bisections as given by \eqref{eq:CombinedEquivRelation}.
Then $\mc R$ is the basis of an \'etale equivalence relation on $\mc R$.
\end{corollary}

\begin{proof}
We appeal to Proposition~\ref{prop:EtaleCombinedRelation}.
Both $\mc R_{\mc B,\mathrm{tail}}$ and $\mc R_{\varphi}$ are \'etale.
Let $\gamma\in \Gamma$ and $x\in X$ such that $(x,x) \in \gamma$.
Using the final statement of Lemma~\ref{lem:TEBRcompatibility}, we have that $\gamma$ contains a bisection $\gamma_0$ in $\mc R_{\mc B,\mathrm{tail}}$.
Since $\mc R_{\mc B,\mathrm{tail}}$ is \'etale, it follows that there is a neighbourhood $U$ of $x$ such that $(y,y)\in\gamma_0\subseteq \gamma$ for all $y\in U$.
\end{proof}

\section{The construction}
\label{sec:Construction}

\subsection{A Bratteli diagram}
\label{sec:ConstructionBratteli}
Let $D$ be a simple dimension group and let $E\subseteq \Q$ be a subgroup which strictly contains $\Z$.
Pick a sequence $(r_\ell)_{\ell=1}^\infty$ of natural numbers such that $r_1>1$, $r_\ell|r_{\ell+1}$ for all $\ell$, and
\begin{equation} E = \Big\{\frac p{r_\ell}: \ell\in\mathbb N\text{ and }p \in \Z\Big\}. \end{equation}
Set 
\begin{equation}\label{eq:d_ell}
d_\ell \coloneqq r_\ell/r_{\ell-1}, \quad\text{ for all $\ell$}. \end{equation} 
Take a Bratteli diagram $\mc B_D=(V_D,E_D)$ (starting at level $\ell_0=0$) such that $H_0(\mc R_{\mc B_D,\mathrm{tail}})\cong D$, and additionally such that, for all $v \in (V_D)_{\ell-1}$ and $w\in (V_D)_{\ell}$,
\begin{equation} \label{eq:EDsize} |E_D(v,w)| \geq r_\ell+d_\ell. \end{equation}
(Recall that $E_D(v,w)$ are the labels on edges from $v$ to $w$, so this says that there are at least $r_\ell+d_\ell$ edges between these vertices.
This can be arranged since $D$ is simple; see the comments after Theorem~2.5 of \cite{Putnam10}.)

We will use the splitting construction from Section \ref{sec:Splitting}; for each $\ell\geq 1$ we pick any vertex $w_a^\ell \in (V_D)_\ell$ to split, and let us label the remaining vertices in $(V_D)_\ell$ as $v_{1,a}^\ell,\dots,v_{k_\ell,a}^\ell$.
Therefore our new Bratteli diagram $\mc C$ with have the level-$\ell$ vertex set
\begin{equation} V_\ell =\{v_{1,a}^\ell,\dots,v_{k_\ell,a}^\ell,w_a^\ell,w_b^\ell\}. \end{equation}
The idea will be that $v_{i,a}^\ell$- and $w_a^\ell$-cylinders represent $0$ in the component $E/\Z$ of the type semigroup, while $w_b^\ell$-cylinders represent $\frac1{r_\ell}+\Z$ in this component.
This motivates a choice of edge splittings which is consistent with this.
We set:
\begin{equation} \label{eq:Fchoice1}
F(v_{i,a}^\ell)\coloneqq \emptyset, \end{equation}
while choosing $F(w_a^\ell),F(w_b^\ell)$ nonempty such that 
\begin{equation} \label{eq:Fchoice2}
|F(w_a^\ell)|= r_{\ell+1},\qquad |F(w_b^\ell)|= d_{\ell+1}.\end{equation}
(This is possible by \eqref{eq:EDsize}.)
We let $\mc C$ be the Bratteli diagram produced by the splitting construction in Section~\ref{sec:Splitting}, with the data just described.
Here are two generic layers of this Bratteli diagram; to avoid clutter we only show the number of split edges (in red).
\begin{equation}
    \begin{tikzcd}[ampersand replacement=\&,row sep=10mm]
    v_{1,a}^\ell \ar[d] \ar[drr] \ar[drrr] \& \cdots \& v_{1,k_\ell}^\ell \ar[dll] \ar[d] \ar[dr] \& w_a^\ell  \ar[dlll] \ar[dl] \ar[d] \& w_b^\ell \ar[d,red,"d_{\ell+1}" description] \ar[dllll] \ar[dll] \ar[dl] \\
    v_{1,a}^{\ell+1} \& \cdots \& v_{1,k_{\ell+1}}^{\ell+1} \& w_a^{\ell+1} \& w_b^{\ell+1} 
    \ar[from=1-4,to=2-5,red,"r_{\ell+1}" description,crossing over]
    \end{tikzcd}
\end{equation}

\subsection{A partial homeomorphism}
Fix some $\ell$ and consider some set $A=\{z_{A,1},\dots,z_{A,r_\ell}\} \subseteq E_{0\dots \ell}$ consisting of $r_\ell$ paths that end at $w_a^\ell$ and some set $B=\{z_{B,1},\dots,z_{B,r_\ell}\} \subseteq E_{0\dots \ell}$ consisting of $r_\ell$ paths that end at $w_b^\ell$.
(Even for $\ell=1$, it is possible to pick such sets, by \eqref{eq:EDsize}).
Assume (for the purpose of checking the constrained condition later) that all paths in $A$ agree on the first $\ell-1$ edges and all paths in $B$ agree on the first $\ell-1$ edges.
We will build a recipe $(A,B,\rho)$ that satisfies the hypotheses of Lemmas~\ref{lem:TEBR} and \ref{lem:TEBRshadow} and Corollary~\ref{cor:EtaleUnderConditions}

First, let us consider $AE_{\ell,\ell+1}$ and $BE_{\ell,\ell+1}$:
\begin{itemize}
    \item Both $AE_{\ell, \ell+1}$ and $BE_{\ell,\ell+1}$ have $r_\ell\big|E_D(w_a,v_{i,a}^{\ell+1})\big|$ paths that end at $v_{i,a}^{\ell+1}$, for $i=1,\dots,k_{\ell+1}$.
    \item $AE_{\ell,\ell+1}$ has $r_\ell\left(\big|E_D(w_a^\ell,w_a^{\ell+1})\big|-r_{\ell+1}\right)$ paths ending at $w_a^{\ell+1}$, while $BE_{\ell,\ell+1}$ has $r_\ell\left(\big|E_D(w_a^\ell,w_a^{\ell+1})\big|-d_{\ell+1}\right)$.
    \item $AE_{\ell,\ell+1}$ has $r_\ell r_{\ell+1}$ paths ending at $w_b^{\ell+1}$, while $BE_{\ell,\ell+1}$ has $r_\ell d_{\ell+1}$.
\end{itemize}

We will use tail-equivalence as much as possible.
The paths that cannot be matched up by tail equivalence consist of $r_\ell(r_{\ell+1}-d_{\ell+1})$ paths in $AE_{\ell,\ell+1}$, all ending at $w_b^{\ell+1}$, and the same number of paths in $BE_{\ell,\ell+1}$, all ending at $w_a^{\ell+1}$.
We will group these into sets $A_i$ and $B_i$ (respectively) of size $r_{\ell+1}$; note that by \eqref{eq:d_ell}, we have
\begin{equation} r_\ell(r_{\ell+1}-d_{\ell+1}) = r_{\ell+1}(r_\ell-1),
\end{equation}
so this is requires $(r_\ell-1)$ such sets.

Let $M_A$ be the set of edges from $w_a^\ell$ to $w_b^{\ell+1}$; then $|M_A|=r_{\ell+1}$.
Since there are $\left(\big|E_D(w_a^\ell,w_a^{\ell+1})\big|-d_{\ell+1}\right) \geq r_{\ell+1}$ edges from $w_b^{\ell}$ to $w_a^{\ell+1}$ (using \eqref{eq:EDsize}), we can also choose a subset $M_B$ of them of size $r_{\ell+1}$.
Then set
\begin{equation} A_i\coloneqq z_{A,i}M_A, \quad B_i\coloneqq z_{B,i}M_B, \quad i=1,\dots,r_\ell-1. \end{equation}

By this construction, $\mathrm{Trunc}_\ell$ induces bijections
\begin{equation} A_i \to M_A, \quad B_i \to M_B, \end{equation}
where we view $M_A,M_B$ as paths of length $1$ in $\mathrm{Trunc}_\ell(\mc C)$.

We may apply the present construction to $(M_B,M_A)$ and $\ell+1$ with the truncated diagram $\mathrm{Trunc}_\ell(\mc C)$; note that all paths in $M_A$ and in $M_B$ agree on the first $0=\ell-1$ edges, so the initial conditions apply.
This yields a recipe $(M_B,M_A,\hat \rho_{\mathrm{mod}})$.
Reversing this recipe (as in Remark~\ref{rmk:ReversingTEBR}) produces a recipe $(M_A,M_B,\rho_{\mathrm{mod}})$, and then using this as a model (i.e., pulling back using $\mathrm{Trunc}_\ell$), we obtain a recipe $(A_i,B_i,\rho_i)$ for $i=1,\dots,r_\ell-1$.

We group the remaining paths in $AE_{\ell,\ell+1}$ and $BE_{\ell,\ell+1}$, into pairs $(w_1,z_1),\dots,(w_m,z_m)$, with each pair ending at the same vertex.
The recipe is then
\begin{equation} \Big(A,B,\rho=\big((w_1,z_1),\dots,(w_m,z_m),(A_1,B_1,\rho_1),\dots,(A_{r_\ell-1},B_{r_\ell-1},\rho_{r_\ell-1})\big)\,\Big). \end{equation}
For the initial recipe we let $\ell=1$ (but we need to allow $\ell\geq 1$ in the description above to allow the construction of the models recipes).

\begin{proposition}
    The above construction yields a recipe which is $(\ell-1)_{\ell=1}^\infty$-constrained, and therefore yields a homeomorphism $\varphi:A X_{\mc C} \to B X_{\mc C}$.
    The recipe satisfies the hypotheses of Corollary~\ref{cor:EtaleUnderConditions}, and therefore the  equivalence relation $\mc R$ generated by $\mc R_{\mc C,\mathrm{tail}}$ and $\varphi$ is \'etale.
    Moreover, the shadow Bratteli diagram homomorphism $T:\mathcal C \to \mathcal B_D$ given by Proposition~\ref{prop:SplittingShadow} satisfies the conditions of Lemma~\ref{lem:TEBRshadow}, and therefore the shadow homeomorphism $s_X:X_{\mc C} \to X_{\mc D}$ induces a surjective homeomorphism 
\begin{equation} \alpha:S(\mc R) \to S(\mc R_{\mc B_D,\mathrm{tail}})\cong D_+. \end{equation}
\end{proposition}

\begin{proof}
By assumption all paths in $A$ agree on the first $\ell-1$ edges and all paths in $B$ agree on the first $\ell-1$ edges, and by the iterative nature of the construction this also holds for descendants, so the recipe is $(\ell-1)_{\ell=1}^\infty$-constrained.

This recipe is highly symmetric, since $(M_A,M_B,\rho_{\mathrm{mod}})$ models each $(A_i,B_i,\rho_i)$, and recursively $(M_A,M_B,\rho_{\mathrm{mod}})$ is highly symmetric.
For every descendant recipe $(A',B',\rho')$, we have that all vertices in $A'$ end at $w_a^{\ell'}$ and all vertices in $B'$ end at $w_b^{\ell'}$, or vice versa; in either case, the set of vertices that $A'$ ends at is disjoint from the set for $B'$.
This shows that the hypotheses of Corollary~\ref{cor:EtaleUnderConditions} holds.

    Finally, we have that each $A_i$ (resp.\ $B_i$) consists of $r_{\ell+1}$ paths ending at $w_b^{\ell+1}$ (resp.\ $w_a^{\ell+1}$).
    Since  $T(w_a^{\ell+1})=w_a^{\ell+1}=T(w_b^{\ell+1})$, it follows that the vertices that the paths in $T(A_i)$ and $T(B_i)$ end at are the same (with the same multiplicity).
    Recursively, this holds also for all descendants.
Hence the hypotheses of Lemma~\ref{lem:TEBRshadow} hold.
\end{proof}

\begin{remark}
When $r_\ell=2$ for all $\ell$, then each stage of the recipe has only one descendant, so the recipe will automatically be highly symmetric.
In this case, we don't need to ensure that all paths in $B_i$ agree on all but the last edge, allowing $\mc B_D$ to have fewer edges.
The examples in Section~\ref{sec:Examples} are constructed in this way.
\end{remark}

We continue to use $\mc R$ and $\alpha$ from the previous proposition.

\begin{proposition}
Let $\mc R$ be the equivalence relation constructed above.
There is a homomorphism $\beta:S(\mc R) \to E/\Z$, satisfying
\begin{equation} 
\beta([wX_{\mc C}]) = \begin{cases} 0,\quad &\text{if $w$ ends at $v_{a,i}^\ell$ or $w_a^\ell$}; \\ \frac1{r_\ell}+\Z,\quad &\text{if $w$ ends at $w_b^\ell$}. \end{cases}
\end{equation}
\end{proposition}

\begin{proof}
We trivially have a map $\hat\beta$ from the set of all cylinder sets in $X_{\mc C}$ to $E/\Z$ given by the above rules.

\noindent    \textbf{Claim}. The map $\hat\beta$ is consistent with cylinder decomposition, i.e.,
    \begin{equation} \label{eq:hatBetaCylinderDecomp} \hat\beta(wX_{\mc C}) = \sum_{z \in wE_{\ell,\ell+1}} \hat\beta(zX_{\mc C}),\quad w \in E_{0\dots\ell}. \end{equation}

    To see this, if $w\in E_{0,\dots,\ell}$ ends at $v_{a,i}^\ell$, then there are no edges from $v_{a,i}^\ell$ to $w_b^{\ell+1}$, and so $wE_{\ell,\ell+1}$ consists only of paths ending at $v_{a,j}^{\ell+1}$s or $w_a^{\ell+1}$.
    Therefore we see that both sides of \eqref{eq:hatBetaCylinderDecomp} are trivially $0$ in this case.

    If $w$ ends at $w_a^\ell$, then there are $|F(w_a^\ell)|=r_{\ell+1}$ paths in $wE_{\ell,\ell+1}$ which end at $w_b^{\ell+1}$ -- call them $z_1,\dots,z_{r_{\ell+1}}$, and the remaining paths all end at $v_{a,j}^{\ell+1}$s or $w_a^{\ell+1}$. 
    Hence we have
    \begin{equation} \begin{split}
        \sum_{z \in wE_{\ell,\ell+1}} \hat\beta(zX_{\mc C})
        &= \sum_{i=1}^{r_{\ell+1}} \hat\beta(z_iX_{\mc C}) \\
        &= r_{\ell+1} \left(\frac1{r_{\ell+1}} + \Z\right) \\
        &= 0 = \hat\beta(wX_{\mc C}),
    \end{split} \end{equation}
    showing \eqref{eq:hatBetaCylinderDecomp} in this case.

    If $w$ ends at $w_b^\ell$, then there are $|F(w_b^\ell)|=d_{\ell+1}$ paths in $wE_{\ell,\ell+1}$ which end at $w_b^{\ell+1}$ -- call them $z_1,\dots,z_{d_{\ell+1}}$, and the remaining paths all end at $v_{a,j}^{\ell+1}$s or $w_a^{\ell+1}$. 
    Hence we have
    \begin{equation} \begin{array}{rcl}
        \sum_{z \in wE_{\ell,\ell+1}} \hat\beta(zX_{\mc C})
        &=& \displaystyle\sum_{i=1}^{d_{\ell+1}} \hat\beta(z_iX_{\mc C}) \\
        &=& d_{\ell+1} \left(\displaystyle\frac1{r_{\ell+1}} + \Z\right) \\
        &\stackrel{\eqref{eq:d_ell}}=& \displaystyle\frac1{r_\ell}+\Z \\
        &=& \hat\beta(wX_{\mc C}),
    \end{array} \end{equation}
    which again verifies \eqref{eq:hatBetaCylinderDecomp}, and finishes the proof of the claim.
\medskip

It follows that $\hat\beta$ extends to a well-defined map, which we will continue to call $\hat\beta$, from the set of all clopen subsets of $X_{\mc C}$ to $E/\Z$, which takes disjoint unions to sums.

To get a well-defined homomorphism from $S(\mc R)$, it suffices to show that $\hat\beta(r(\gamma))=\hat\beta(s(\gamma))$ for any compact open bisection $\gamma\subseteq \mc R$.
Since we can decompose any such compact open bisection into ones coming from the basis described in \eqref{eq:CombinedEquivRelation}, it suffices to prove this assuming that either $\gamma \subseteq \mc R_{\mc C,\mathrm{tail}}$ or $\gamma = \{(x,\varphi(x)): x\in K\}$ for some clopen set $K\subseteq A$.

Since the definition of $\hat\beta(wX_{\mc C})$ depends only on the vertex that $w$ ends at, we see that $\hat\beta(r(\gamma))=\hat\beta(s(\gamma))$ when $\gamma \subseteq \mc R_{\mc C,\mathrm{tail}}$.

For the other case, there is some $\ell$ such that $K$ is a disjoint union of cylinders defined by paths of length $\ell$.
Since $\rho$ is $(\ell-1)_{\ell=1}^\infty$-constrained, $K$ can then be written as a disjoint union of sets where $\varphi$ acts by tail-equivalence, together with sets of the form $A'X_{\mc C}$ where $(A',B',\rho')$ is a descendant of $\rho$.
Since we have already handled tail-equivalence, we may assume that $K$ has the form $A'X_{\mc C}$ where $(A',B',\rho')$ is a descendant of $\rho$.
In this case, $s(\gamma)=A'X_{\mc C}$ and $r(\gamma)=B'X_{\mc C}$.
Either $A'$ (resp.\ $B'$) consists of $r_{\ell'}$ paths ending at $w_a^{\ell'}$ (resp.\ $w_b^{\ell'}$) or vice versa; we assume the former.
Then we have
\begin{equation} \begin{split}
    \hat\beta(r(\gamma)) &= \hat\beta(B'X_{\mc C}) \\
    &= \sum_{w \in B'} \hat\beta(wX_{\mc C}) \\
    &= \sum_{w\in B'} \frac1{r_{\ell'}}+\Z \\
    &= r_{\ell'} \left(\frac1{r_{\ell'}}+\Z\right) \\
    &= 0 \\
    &= \sum_{w \in A'} \hat\beta(wX_{\mc C}) \\
    &= \hat\beta(s(\gamma)).
\end{split}\end{equation}
This concludes the proof that $\hat\beta$ factors through $S(\mc R)$, and thereby induces the map $\beta$ as described.
\end{proof}

\begin{theorem}
The homomorphisms $\alpha$ and $\beta$ combine to give an isomorphism
\begin{equation} (\alpha,\beta):S(\mc R) \to \{(0,0)\} \cup ((D_+\setminus \{0\}) \oplus (E/\Z)). \end{equation}
Consequently, $H_0(\mc R) \cong D \oplus (E/\Z)$, with $H_0(\mc R)_+$ identifying with $\{(0,0)\} \cup ((D_+\setminus\{0\}) \oplus (E/\Z))$.
\end{theorem}

\begin{proof}
The final statement follows immediately by Proposition~\ref{prop:HomologyEnvelopingTypeSemigroup}.

\textbf{Injectivity}.
Suppose that $x,y \in S(\mc R)$ are such that $(\alpha,\beta)(x)=(\alpha,\beta)(y)$.
Write $x=\sum_{i=1}^k [p_iX_{\mc C}]$ and $y= \sum_{i=1}^{k'} [q_iX_{\mc C}]$ for some paths $p_1,\dots,p_k,q_1,\dots,q_{k'}$, and we may in fact assume that all these paths are of the same length: $p_1,\dots,p_k,q_1,\dots,q_{k'} \in E_{0\dots\ell'}$ for some $\ell'$.

Since $\alpha$ is induced by the shadow Bratteli diagram homomorphism $T:\mc C \to \mc B_D$, we have that
\begin{equation} \sum_{i=1}^k [T(p_i)X_{\mc B_D}] = \sum_{i=1}^{k'} [T(q_i)X_{\mc B_D}] \end{equation}
in $S(\mc R_{\mc B_D,\mathrm{tail}}) = D_+$.
Therefore, by possibly going to a further level (i.e., subdividing each cylinder set) and by possibly permuting the elements, we have $k=k'$ and we may assume that $T(p_i)$ ends at the same vertex as $T(q_i)$ for each $i$.

If $T(p_i)$ ends at $v_{a,j}^{\ell'}$ for some $j$, then since this vertex has a single preimage under $T$, $p_i$ and $q_i$ would end at the same vertex, and so $[p_iX_{\mc C}] = [q_iX_{\mc C}]$.
We may therefore cancel these terms out, so that what remains are only $p_i$ ending at either $w_a^{\ell'}$ or $w_b^{\ell'}$.
We may likewise pair up any $p_i$ ending at $w_a^{\ell'}$ or $w_b^{\ell'}$ with any $q_i$ ending at the same vertex and cancel their terms.
Upon doing so, we have that all the $p_i$ end at the same vertex -- either $w_a^{\ell'}$ or $w_b^{\ell'}$ -- and all the $q_i$ end at a different vertex.
So we may assume without loss of generality that all $p_i$ end at $w_b^{\ell'}$ and all $q_i$ end at $w_a^{\ell'}$.

We now compute
\begin{equation} \begin{split}
    0 &= \beta(y) \\
    &= \beta(x) \\
    &= \sum_{i=1}^k \beta([p_iX_{\mc C}]) \\
    &= \sum_{i=1}^k \frac1{r_{\ell'}}+\Z \\
    &= \frac{k}{r_\ell'}+\Z.
\end{split} \end{equation}
Therefore, we see that 
$k=r_{\ell'}\hat k$ for some natural number $\hat k$.

We may find a descendant $(A',B',\rho')$ at stage $\ell'$. 
Then $A'$ consists of $r_{\ell'}$ paths ending at $w_a^{\ell'}$ and $B'$ consists of $r_{\ell'}$ paths ending at $w_b^{\ell'}$, or vice versa -- but for simplicity assume not.
Since this is a descendant of $\rho$, and $\varphi$ fits $\rho$, we have $\varphi(A'X_{\mc C}) = B'X_{\mc C}$.
Then
\begin{equation}
\begin{split}
\sum_{i=1}^k [p_i X_{\mc C}]
&= \hat k[A'X_{\mc C}]   \\
&= \hat k[B'X_{\mc C}] \\
&= \sum_{i=1}^k [q_i X_{\mc C}],
\end{split}
\end{equation}
by using tail-equivalence for the first and last equalities, and using $\varphi(A'X_{\mc C})=B'X_{\mc C}$ in the middle.
This shows that $x=y$, as required.

\textbf{Surjectivity}.
Let $x \in D_+\setminus \{0\}$ and consider an arbitrary element $y+\Z$ of $E/\Z$, where $y \in E$ and $0 \leq y < 1$.
Since $x$ is in $D_+$, we can write it as 
\begin{equation} x=\sum_{i=1}^k [p_iX_{\mc B_D}], \end{equation}
for some paths $p_1,\dots,p_k$ in $\mc B_D$.
By possibly subdividing, we may assume that they have length $\ell'$ such that $r_{\ell'} y \in \mathbb N$.
By possibly further subdividing (using \eqref{eq:EDsize},\eqref{eq:Fchoice1},\eqref{eq:Fchoice2}), we may assume that at least $r_{\ell'} y$ of them end at $w_a^{\ell'}$.
(This is possible because when  we divide $k$ cylinders from paths of length $\ell_1$ to paths of length $\ell_1+1$, then we end up with at least $kr_{\ell_1+1}$ paths ending at $w_a^{\ell_1}$.)

By reordering, we may assume that each of $p_1,\dots,p_t$ ends at $w_a^{\ell'}$ and that each of $p_{t+1},\dots,p_k$ ends at some $v_{j,a}^{\ell'}$ (for some $t\geq r_{\ell'}y$).
Then pick paths $q_1,\dots,q_k$ of length $\ell'$ in $\mc C$, such that:
\begin{itemize}
    \item $q_i$ ends at $w_{b}^{\ell'}$, for $i=1,\dots, r_{\ell'}y$;
    \item $q_i$ ends at $w_{a}^{\ell'}$, for $i=r_{\ell'}y+1,\dots,t$;
    \item $q_i$ ends at $v_{j,a}^{\ell'}$ when $p_i$ ends at $v_{j,a}^{\ell'}$, for $i=t+1,\dots,k$.
\end{itemize}
Set 
\begin{equation} z\coloneqq [q_1X_{\mc C}]+\cdots+[q_kX_{\mc C}] \in S(\mc R). \end{equation}

We have that $T(q_i)$ ends at the same vertex as $p_i$, so that $\alpha([q_iX_{\mc C}])=[p_iX_{\mc B_D}]$ for all $i$.
Hence
\begin{equation} \alpha(z) = \sum_{i=1}^k [p_iX_{\mc B_D}] = x. \end{equation}
We also have 
\begin{equation} \beta([q_iX_{\mc C}])=\begin{cases} \frac1{r_{\ell'}}+\Z, \quad &\text{for $i=1,\dots,r_{\ell'}y$}; \\ 0,\quad &\text{otherwise}. \end{cases} \end{equation}
Therefore,
\begin{equation} \beta(z) = \sum_{i=1}^k \beta([q_iX_{\mc C}]) = \frac{r_{\ell'}y}{r_{\ell'}}+\Z = y+\Z, \end{equation}
and so $(\alpha,\beta)(z)=(x,y+\Z)$ as required.
\end{proof}

\section{Approximately inner flip}
\label{sec:AIF}

Here we define a notion of approximately inner flip for \'etale equivalence relations, inspired by a notion of the same name introduced by Effros and Rosenberg for C*-algebras (\cite{EffrosRosenberg}).

Recall that the \emph{topological full group} of an \'etale equivalence relation $\mc R$ with compact totally disconnected unit space, denoted $[[\mc R]]$, consists of homeomorphisms $\alpha$ of $\mc R^{(0)}$ such that the set $\{(x,\alpha(x)):x\in \mc R^{(0)}\}$ is a compact open bisection in $\mc R$.

\begin{definition}
Let $\mathcal R$ be an \'etale equivalence relation with compact metrizable totally disconnected unit space.
We say that $\mathcal R$ has \emph{approximately inner flip} if there is a sequence of maps $(\alpha_n)_n$ in $[[\mathcal R \times \mathcal R]]$ such that $\alpha_n \to \sigma$ uniformly on $\mc R^{(0)} \times \mc R^{(0)}$, where $\sigma:\mc R^{(0)}\times \mc R^{(0)} \to \mc R^{(0)} \times \mc R^{(0)}$ is given by $\sigma(x,y)\coloneqq(y,x)$.
\end{definition}

\begin{example}
Tail equivalence on the space of infinite words in a finite alphabet has approximately inner flip.
\end{example}

Since $\mc R^{(0)}$ is compact, it follows in the above definition that $\alpha_n^{-1}$ also converges uniformly to $\sigma$.

The following shows how the above property relates to the C*-algebraic notion.
We refer to \cite[Section 5.6]{BrownOzawa} for the definition of the reduced C*-algebra $C^*_r(\mathcal G)$ associated to an \'etale groupoid $\mathcal G$.

\begin{proposition}
Let $\mathcal R$ be an \'etale equivalence relation with compact metrizable totally disconnected unit space.
It has approximately inner flip iff there is a sequence of unitary normalizers $(u_n)_n$ in $C^*_r(\mathcal R)\otimes C^*_r(\mathcal R)$ (that is, $u_n^*(C(\mc R^{(0)}) \otimes C(\mc R^{(0)}))u_n = C(\mc R^{(0)})\otimes C(\mc R^{(0)})$ for all $n$) such that
\begin{equation} \|u_n(f \otimes g)u_n^* - g \otimes f\| \to 0, \quad f,g \in C(\mc R^{(0)}). \end{equation}
\end{proposition}

\begin{proof}
Note that $C^*_r(\mathcal R) \otimes C^*_r(\mathcal R) \cong C^*_r(\mathcal R \times \mathcal R)$, with the isomorphism taking $C(\mc R^{(0)})\otimes C(\mc R^{(0)})$ to $C(\mc R^{(0)} \times \mc R^{(0)})$ (see \cite[Lemma 2.10 and its proof]{AustinMitra}).

Suppose that $\mathcal R$ has approximately inner flip and let $(\alpha_n)_n$ be the maps given by the definition.
Each $\alpha_n$ gives rise to a unitary normalizer $u_n$ in $C^*_r(\mathcal R \times \mathcal R)$ (which we identify with $C^*_r(\mathcal R) \otimes C^*_r(\mathcal R)$), such that
\begin{equation} u_nFu_n^* = F \circ \alpha_n. \end{equation}
Since $\alpha_n$ converges uniformly to $\sigma$, it follows that $u_n(f\otimes g)u_n^*$ converges in norm to $(f \otimes g) \circ \sigma = g \otimes f$.

Conversely, suppose we are given unitary normalizers $(u_n)_n$ such that $u_n(f\otimes g)u_n^*$ converges in norm to $g\otimes f$, for all $f,g \in C(\mc R^{(0)})$.
Then by \cite[Proposition 4.7]{Renault08} (based on \cite[Proposition 1.6]{Kumjian86}), there exists $\alpha_n \in [[\mathcal R \times \mathcal R]]$ such that $u_nFu_n^* = F \circ \alpha_n$.
Since $(f \otimes g) \circ \alpha_n$ converges to $g \otimes f = (f\otimes g) \circ \sigma$, for all $f,g \in C(\mc R^{(0)})$, it follows that $F \circ \alpha_n \to F \circ \sigma$ for all $F \in C(\mc R^{(0)} \times \mc R^{(0)})$, and consequently, $\alpha_n \to \sigma$ uniformly.
\end{proof}

Here is another characterization of approximately inner flip.

\begin{proposition}
    Let $\mathcal R$ be an \'etale equivalence relation with compact metrizable totally disconnected unit space.
    Then $\mc R$ has approximately inner flip if and only if, for all clopen sets $A,B\subseteq \mc R^{(0)}$, we have $[A\times B]=[B\times A]$ in $S(\mc R\times\mc R)$.
\end{proposition}

\begin{proof}
    Suppose that $\mathcal R$ has approximately inner flip, and obtain the sequence $(\alpha_n)_{n=1}^\infty$ in $[[\mc R\times \mc R]]$ from the definition.
    Then for all $n$ sufficiently large, we must have $\alpha_n(A\times B) = B \times A$, and therefore $[A\times B]=[B\times A]$ in $S(\mc R\times \mc R)$.

For the converse, it is useful to fix a metric on $\mc R^{(0)}$.
Given $\epsilon>0$ we may decompose $X=A_1\sqcup \cdots \sqcup A_n$ where each $A_i$ is clopen and has diameter at most $\epsilon$.
For each $i,j$, the hypothesis tells us that there exists a compact open bisection $\gamma_{i,j}\subseteq \mc R\times\mc R$ such that $s(\gamma_{i,j})=A_i\times B_j$ and $r(\gamma_{i,j})=B_j\times A_i$.
We assign to $\gamma_{i,j}$ the homeomorphism $\alpha_{i,j}:A_i \times B_j \to B_j \times A_i$, such that
\begin{equation} \gamma_{i,j} = \{((x,y),\alpha_{i,j}(x,y))\}. \end{equation}
Define $\alpha:\mc R^{(0)}\times \mc R^{(0)} \to \mc R^{(0)}\times \mc R^{(0)}$ such that 
\begin{equation} \alpha|_{A_i \times B_j} = \alpha_{i,j} \quad\text{for all $i,j$}. \end{equation}
Then we have $\alpha \in [[\mc R\times \mc R]]$ and $d(\alpha(x,y),(x,y)) < \epsilon$ for all $x,y \in \mc R^{(0)}$.
\end{proof}

The following is straightforward.

\begin{proposition}
    \label{prop:TensorS}
    Let $\mc G_1,\mc G_2$ be \'etale groupoids with totally disconnected unit spaces.
    There is an isomorphism $S(\mc G_1\times \mc G_2)\cong S(\mc G_1) \otimes S(\mc G_2)$ taking $[A\times B]$ to $[A]\otimes[B]$ for compact open sets $A\subseteq \mc G_1^{(0)}$ and $B\subseteq\mc G_2^{(0)}$.
\end{proposition}

\begin{proof}
For compact open sets $A\subseteq \mc G_!^{(0)}$ and $B\subseteq\mc G_2^{(0)}$, we can see that the class of $[A\times B]$ in $S(\mc G_1\times \mc G_2)$ only depends on the classes $[A]$ and $[B]$ in $S(\mc G_1)$ and $S(\mc G_2)$ respectively.
Since $[A] \mapsto [A \otimes B]$ (with $B$ fixed) and $[B] \mapsto [A \otimes B]$ (with $A$ fixed) are homomorphisms, this gives rise to a well-defined homomorphism $\theta:S(\mathcal G_1) \otimes S(\mathcal G_2) \to S(\mathcal G_1 \times \mathcal G_2)$.

We will argue that its inverse of $\theta$ is well-defined, in order to show it is bijective.
As mentioned before Definition~\ref{def:TypeSemigroup}, $C_0(\mc G_1^{(0)},\mathbb Z_+)$ identifies with the universal abelian semigroup generated by all $\langle A\rangle$ where $A$ is a compact open subset of $\mc G_1^{(0)}$, subject to $\langle A\sqcup B \rangle = \langle A\rangle + \langle B \rangle$.
We have a natural isomorphism $C_0(\mc G_1^{(0)}\times \mc G_2^{(0)},\Z_+) \cong C_0(\mc G_1^{(0)},\Z_+) \otimes C_0(\mc G_2^{(0)},\Z_+)$, and this induces a homomorphism $\psi:C_0(\mc G_1^{(0)}\times \mc G_2^{(0)},\Z_+) \to S(\mc G_1) \otimes S(\mc G_2)$, taking  $\langle A\times B\rangle$ to $[ A] \otimes  [B]$.

To check that $\psi$ factors through $S(\mc G_1\times \mc G_2)$, we need to check that if $\gamma$ is a compact open bisection in $\mc G_1\times \mc G_2$ then $\psi(\langle s(\gamma)\rangle)=\psi(\langle r(\gamma)\rangle)$.
We can decompose $\gamma$ into a finite union of rectangles:
\begin{equation} \gamma = (\gamma_1^{(1)}\times \gamma_1^{(2)}) \sqcup \cdots \sqcup (\gamma_k^{(1)}\times \gamma_k^{(2)}), \end{equation}
where each $\gamma_i^{(j)}\subseteq \mc G_j$ is a compact open bisection.
Therefore,
\begin{equation}\begin{split}
    \psi(\langle s(\gamma)\rangle) 
    &= \sum_{i=1}^k \psi(\langle s(\gamma_i^{(1)}\times \gamma_i^{(2)})\rangle)  \\
    &= \sum_{i=1}^k \psi(\langle s(\gamma_i^{(1)}) \times s(\gamma_i^{(2)})\rangle)  \\
    &= \sum_{i=1}^k [s(\gamma_i^{(1)})] \otimes [s(\gamma_i^{(2)})]  \\
    &= \sum_{i=1}^k [r(\gamma_i^{(1)})] \otimes [r(\gamma_i^{(2)})]  \\
    &= \psi(\langle r(\gamma)\rangle),
\end{split} \end{equation}
as required, where in the last step we essentially do the first three steps in reverse.
\end{proof}

This allows us to characterize approximately inner flip in terms of the type semigroup.

\begin{corollary}\label{cor:aifS}
    Let $\mathcal R$ be an \'etale equivalence relation with compact metrizable totally disconnected unit space.
    Then $\mc R$ has approximately inner flip if and only if $x\otimes y=y\otimes x$ in $S(\mc R)\otimes S(\mc R)$, for all $x,y \in S(\mc R)$.
\end{corollary}

\begin{example}
    Let $D$ be a non-cyclic subgroup of $\mathbb Q$.
    Let $E$ be a subgroup of $\mathbb Q$, such that whenever $\frac1r \in E$, we have that $D$ is $r$-divisible.
    Let $\mc R$ be an \'etale equivalence relation on a totally disconnected space given by our main construction, such that
    \begin{equation} S(\mc R) \cong \big((D_+\setminus 0) \oplus E/\Z\big) \cup \{(0,0)\} .\end{equation}
    Let us show that $\mc R$ has approximately inner flip, by showing that $S(\mc R)$ satisfies the condition in Corollary~\ref{cor:aifS} above.

    Consider nonzero elements $x,y \in S(\mc R)$, with $x=(x_1,x_2+\Z), y=(y_1,y_2+\Z)$, where $x_1,y_1\in D_+\setminus\{0\}$ and $x_2,y_2\in E$.
    Finding a common denominator for $x_2,y_2$, and then finding some sufficiently small $\hat c \in D_+$, we may write
    \begin{equation} x=(a,0)+p\Big(\hat c,\frac1r+\Z\Big),\quad x=(b,0)+q\Big(\hat c,\frac1r+\Z\Big), \end{equation}
    where $\frac1r \in E$, $\hat a,\hat b,\in D_+\setminus\{0\}$ and $p,q\in \Z_+$.
    Using common denominators again, we write $\hat a=\frac as$ and $\hat b=\frac bs$, where $\frac1s\in D$.
    We compute
    \begin{equation} \begin{split}
        \Big(\frac as,0\Big)\otimes \Big(\frac bs,0\Big) &= ab\Big(\frac1s,0\Big)\otimes \Big(\frac1s,0\Big), \\
        \Big(\frac as,0\Big)\otimes \Big(\hat c,\frac1r+\Z\Big) &= \Big(\frac a{sr},0\Big)\otimes \Big(r\hat c,\frac rr+\Z\Big) \\
        &= 0, \\
        \Big(\hat c,\frac1r\Big)\otimes \Big(\frac bs,0\Big) &= 0,\text{ similarly}.
    \end{split}\end{equation}
    Therefore,
    \begin{equation} \begin{split}
        x \otimes y = ab\Big(\frac1s,0\Big)\otimes \Big(\frac1s,0\Big) + pq\Big(\hat c, \frac1r+\Z\Big)\otimes \Big(\hat c, \frac1r+\Z\Big) = y \otimes x.
    \end{split}\end{equation}
\end{example}


\begin{thebibliography}{10}

\bibitem{AraBonickeBosaLi}
Pere Ara, Christian B\"{o}nicke, Joan Bosa, and Kang Li.
\newblock The type semigroup, comparison, and almost finiteness for ample
  groupoids.
\newblock {\em Ergodic Theory Dynam. Systems}, 43(2):361--400, 2023.

\bibitem{AustinMitra}
Kyle Austin and Atish Mitra.
\newblock Groupoid models of {$C^*$}-algebras and the {G}elfand functor.
\newblock {\em New York J. Math.}, 27:740--775, 2021.

\bibitem{BonickeLi}
Christian B\"{o}nicke and Kang Li.
\newblock Ideal structure and pure infiniteness of ample groupoid
  {$C^*$}-algebras.
\newblock {\em Ergodic Theory Dynam. Systems}, 40(1):34--63, 2020.

\bibitem{Bratteli72}
Ola Bratteli.
\newblock Inductive limits of finite dimensional {C*}-algebras.
\newblock {\em Trans. Amer. Math. Soc.}, 171:195--234, 1972.

\bibitem{BrownOzawa}
Nathanial~P. Brown and Narutaka Ozawa.
\newblock {\em {C*}-algebras and finite-dimensional approximations}, volume~88
  of {\em Graduate Studies in Mathematics}.
\newblock American Mathematical Society, Providence, RI, 2008.

\bibitem{CrainicMoerdijk}
Marius Crainic and Ieke Moerdijk.
\newblock A homology theory for \'{e}tale groupoids.
\newblock {\em J. Reine Angew. Math.}, 521:25--46, 2000.

\bibitem{EffrosHandelmanShen}
Edward~G. Effros, David~E. Handelman, and Chao~Liang Shen.
\newblock Dimension groups and their affine representations.
\newblock {\em Amer. J. Math.}, 102(2):385--407, 1980.

\bibitem{EffrosRosenberg}
Edward~G. Effros and Jonathan Rosenberg.
\newblock {C*}-algebras with approximately inner flip.
\newblock {\em Pacific J. Math.}, 77(2):417--443, 1978.

\bibitem{Elliott76}
George~A. Elliott.
\newblock On the classification of inductive limits of sequences of semisimple
  finite-dimensional algebras.
\newblock {\em J. Algebra}, 38(1):29--44, 1976.

\bibitem{ForrestHuntonKellendonk}
A.~H. Forrest, J.~R. Hunton, and J.~Kellendonk.
\newblock Cohomology of canonical projection tilings.
\newblock {\em Comm. Math. Phys.}, 226(2):289--322, 2002.

\bibitem{GahlerHuntonKellendonk}
Franz G\"{a}hler, John Hunton, and Johannes Kellendonk.
\newblock Integral cohomology of rational projection method patterns.
\newblock {\em Algebr. Geom. Topol.}, 13(3):1661--1708, 2013.

\bibitem{GiordanoPutnamSkau95}
Thierry Giordano, Ian~F. Putnam, and Christian~F. Skau.
\newblock Topological orbit equivalence and {C*}-crossed products.
\newblock {\em J. Reine Angew. Math.}, 469:51--111, 1995.

\bibitem{Goodearl:book}
K.~R. Goodearl.
\newblock {\em Partially ordered abelian groups with interpolation}, volume~20
  of {\em Mathematical Surveys and Monographs}.
\newblock American Mathematical Society, Providence, RI, 1986.

\bibitem{HermanPutnamSkau}
Richard~H. Herman, Ian~F. Putnam, and Christian~F. Skau.
\newblock Ordered {B}ratteli diagrams, dimension groups and topological
  dynamics.
\newblock {\em Internat. J. Math.}, 3(6):827--864, 1992.

\bibitem{Krieger79}
Wolfgang Krieger.
\newblock On a dimension for a class of homeomorphism groups.
\newblock {\em Math. Ann.}, 252(2):87--95, 1979/80.

\bibitem{Kumjian86}
Alexander Kumjian.
\newblock On {C*}-diagonals.
\newblock {\em Canad. J. Math.}, 38(4):969--1008, 1986.

\bibitem{Matui08}
Hiroki Matui.
\newblock Torsion in coinvariants of certain {C}antor minimal {$\Bbb
  Z^2$}-systems.
\newblock {\em Trans. Amer. Math. Soc.}, 360(9):4913--4928, 2008.

\bibitem{Matui12}
Hiroki Matui.
\newblock Homology and topological full groups of \'{e}tale groupoids on
  totally disconnected spaces.
\newblock {\em Proc. Lond. Math. Soc. (3)}, 104(1):27--56, 2012.

\bibitem{Putnam10}
Ian~F. Putnam.
\newblock Orbit equivalence of {C}antor minimal systems: a survey and a new
  proof.
\newblock {\em Expo. Math.}, 28(2):101--131, 2010.

\bibitem{Putnam18}
Ian~F. Putnam.
\newblock Some classifiable groupoid {$C^*$}-algebras with prescribed
  {$K$}-theory.
\newblock {\em Math. Ann.}, 370(3-4):1361--1387, 2018.

\bibitem{RainoneSims}
Timothy Rainone and Aidan Sims.
\newblock A dichotomy for groupoid {$C^\ast$}-algebras.
\newblock {\em Ergodic Theory Dynam. Systems}, 40(2):521--563, 2020.

\bibitem{Renault08}
Jean Renault.
\newblock Cartan subalgebras in {C*}-algebras.
\newblock {\em Irish Math. Soc. Bull.}, (61):29--63, 2008.

\bibitem{SimsSzaboWilliams}
Aidan Sims, G\'{a}bor Szab\'{o}, and Dana Williams.
\newblock {\em Operator algebras and dynamics: groupoids, crossed products, and
  {R}okhlin dimension}.
\newblock Advanced Courses in Mathematics. CRM Barcelona.
  Birkh\"{a}user/Springer, Cham, [2020] \copyright 2020.
\newblock Lecture notes from the Advanced Course held at Centre de Recerca
  Matem\`atica (CRM) Barcelona, March 13--17, 2017.

\bibitem{Vershik81}
A.~M. Vershik.
\newblock Uniform algebraic approximation of shift and multiplication
  operators.
\newblock {\em Dokl. Akad. Nauk SSSR}, 259(3):526--529, 1981.

\bibitem{Vershik82}
A.~M. Ver\v{s}ik.
\newblock A theorem on periodical {M}arkov approximation in ergodic theory.
\newblock In {\em Ergodic theory and related topics ({V}itte, 1981)}, volume~12
  of {\em Math. Res.}, pages 195--206. Akademie-Verlag, Berlin, 1982.

\end{thebibliography}
\end{document}